\documentclass[final,1p,times]{elsarticle}

\usepackage{graphicx}

\usepackage{amssymb}
\usepackage{amsthm}
\usepackage{amsmath,amssymb,amsopn,amsfonts,mathrsfs,amsbsy,amscd}
\usepackage{longtable}
\usepackage{caption}
\usepackage{multirow}
\usepackage{xcolor}

\newcommand{\ass}{\mathrm{ass}}

\newcommand{\too}{\longrightarrow}
\newcommand{\om}{\omega}
\newcommand{\esp}{\quad\mbox{and}\quad}

\def\br{[\;,\;]}
\newcommand{\A}{{\cal A}}

\newcommand{\G}{\mathfrak{g}}
\newcommand{\g}{\mathfrak{g}}

\newcommand{\h}{{\mathfrak{h}}}
\newcommand{\ad}{{\mathrm{ad}}}

\newcommand{\tr}{{\mathrm{tr}}}

\newcommand{\Li}{{\mathrm{L}}}
\newcommand{\Ri}{{\mathrm{R}}}

\newcommand{\Lei}{\mathrm{Leib}}
\newcommand{\Lie}{\mathrm{Lie}}

\newcommand{\Om}{\Omega}

\newcommand{\al}{\alpha}
\newcommand{\be}{\beta}

\newcommand{\e}{\epsilon}

\newcommand{\la}{\lambda}

\newtheorem{Def}{Definition}[section]
\newtheorem{theo}{Theorem}[section]
\newtheorem{pr}{Proposition}[section]
\newtheorem{Le}{Lemma}[section]
\newtheorem{co}{Corollary}[section]

\newtheorem{exem}{Example}

\font\bb=msbm10

\def\K{\hbox{\bb K}}

\def\R{\hbox{\bb R}}

\def\S{\hbox{\bb S}}

\begin{document}

\begin{frontmatter}
	
	
	
	\title{ Symplectic Leibniz algebras as a non-commutative version of symplectic Lie algebras  }
	
	
	\author[label1]{ Fatima-Ezzahrae Abid  }
	\address[label1]{Universit\'e Cadi-Ayyad\\
		Facult\'e des sciences et techniques\\
		BP 549 Marrakech Maroc\\e-mail: abid.fatimaezzahrae@gmail.com}
	
	\author[label2]{Mohamed Boucetta}
	\address[label2]{Universit\'e Cadi-Ayyad\\
		Facult\'e des sciences et techniques\\
		BP 549 Marrakech Maroc\\e-mail: m.boucetta@uca.ac.ma}
	
	
	
	
	
	\begin{abstract} We introduce symplectic left Leibniz algebras and symplectic right Leibniz algebras as generalizations of symplectic Lie algebras.  These algebras possess a left symmetric product and are Lie-admissible. We describe completely symmetric Leibniz algebras that are symplectic as both left and right Leibniz algebras. Additionally, we show that symplectic left or right Leibniz algebras can be constructed from a symplectic Lie algebra and a vector space through a method that combines the double extension process and the $T^*$-extension. This approach allows us to generate a broad class of examples. 	
	\end{abstract}

\end{frontmatter}

{\it Keywords: Leibniz algebras, symplectic Lie algebras,  double extension.}

\section{Introduction}\label{section1}

Symplectic  Lie algebras and symplectic Lie groups are fundamental concepts in symplectic geometry, an important area in mathematics and theoretical physics. Symplectic Lie algebras were studied intensively by many authors and it will be tedious to cite all the papers on this topic and so one can consult the survey \cite{cortes} for a detailed bibliography. In this paper, we introduce symplectic left Leibniz algebras and symplectic right Leibniz algebras as generalizations of symplectic Lie algebras. 

A symplectic Lie algebra is a Lie algebra $(\G,\br)$ endowed with a nondegenerate bilinear skew-symmetric form $\om$ such that
\[ \om([X,Y],Z)+\om([Y,Z],X)+\om([Z,X],Y)=0 \]for any $X,Y,Z\in\G$. The key point, in order to generalize this concept to  Leibniz algebras, is that this relation is equivalent to
\[ X\star Y-Y\star X=[X,Y] \]for any $X,Y\in\G$ where $\star$ is the product given by $X\star Y=-\ad_X^*Y$ where $\ad_XY=[X,Y]$ and $\ad_X^*$ is the adjoint of $\ad_X$ with respect to $\om$. Moreover, the Jacobi identity $\ad_{[X,Y]}=[\ad_X,\ad_Y]$ is equivalent to $\star$ is a left symmetric product, i.e.,
\[ \ass(X,Y,Z)=\ass(Y,X,Z) \]and $\ass(X,Y,Z)=(X\star Y)\star Z-X\star(Y\star Z)$.   \\
{Leibniz algebras} are  a non-commutative generalization of Lie algebras and were first introduced and investigated in the papers of  Bloh \cite{bloh, bloh1} under the name of D-algebras.
Then they were  rediscovered by  Loday \cite{loday} who called them  Leibniz algebras. A left Leibniz algebra  is an algebra $(\A,\bullet)$ over a field $\mathbb{K}$ such that, for any $X\in\G$, $\Li_X$  is a derivation of $(\A,\bullet)$,  where $\Li_XY=X\bullet Y$, i.e.,
\begin{equation}\label{L} \Li_{X\bullet Y}=[\Li_X,\Li_Y],\quad X,Y\in\A. 
\end{equation}
Right Leibniz algebras are defined in a similar way by considering right multiplications instead of left multiplications\footnote{The two notions of left and right Leibniz algebras are equivalent and, in this paper, we choose to work with left Leibniz algebras. }.
An algebra which is both left and right Leibniz is called symmetric Leibniz algebra. A Lie algebra is obviously a symmetric Leibniz algebra.
Many results of the theory of Lie algebras can be extended to left Leibniz algebras \cite{nil}. Moreover, real left Leibniz algebras are the infinitesimal version of Lie racks. In 2004, Kinyon \cite{kinyon} proved that if $(\mathcal{R}, e)$ is a pointed Lie rack, $T_e\mathcal{R}$ carries a structure of left Leibniz algebra.

Let us introduce the notion of symplectic left Leibniz algebra. For a left Leibniz algebra $(\A,\bullet)$, put
\[ X\bullet Y=[X,Y]+X\diamond Y, \]where
\[ [X,Y]=\frac12(X\bullet Y-Y\bullet X)\esp 
X\diamond Y=\frac12(X\bullet Y+Y\bullet X). \]
We call a nondegenerate bilinear skew-symmetric form $\om$ on $\A$ symplectic if for any $X,Y\in\A$,
\begin{equation}
	X\star Y-Y\star X=[X,Y]\esp X\star Y=-\Li_X^*Y,
\end{equation} where $\Li_X^*$ is the adjoint of $\Li_X$ with respect to $\om$. We call the triple $(\A,\bullet,\om)$ {\it symplectic left Leibniz algebra}. This definition clearly generalizes the notion of symplectic Lie algebras. More importantly, the relation \eqref{L} implies that $\star$
is a left symmetric product, making $(\A,\bullet)$
Lie-admissible. With this definition in mind, the purpose of this paper is to undertake a comprehensive study of symplectic left Leibniz algebras. It is worth noting that a concept of symplectic left Leibniz algebra, which involves a nondegenerate bilinear symmetric form instead of a skew-symmetric one, was introduced in \cite{Tang}. This concept, actually, generalizes quadratic Lie algebras rather than symplectic ones. Our choice of terminology is therefore well justified.

Let us now, enumerate briefly our main results and the organization of the paper.
\begin{enumerate}
	\item In Section \ref{section2}, we state clearly the definition of symplectic left (resp. right) Leibniz algebras  and we give their fundamental properties.  We associate, naturally, to any symplectic left Leibniz algebra $(\A,\bullet,\om)$ a symplectic Lie algebra $(\G,\br_\G,\om_\G)$ and a trivial algebra $\h$ and we call them the core of $(\A,\bullet,\om)$.
	\item A symmetric Leibniz algebra is {\it bi-symplectic} if it carries a nondegenerate bilinear skew-symmetric form which is
	symplectic for both the left and the right Leibniz structure. In Section \ref{section3}, we give a complete description of bi-symplectic symmetric Leibniz algebras.
	\item  Section \ref{section4} is devoted to the proof of our main result (see Theorem \ref{main}). It states that any symplectic left Leibniz algebra can be built from its core by a method which combines the process of double extension introduced by Medina-Revoy and the $T^*$-extension introduced by Bordemann (see \cite{Medina, bor}).  This result allows us to get a large classes  of symplectic left Leibniz algebras (see  Section \ref{section5}).
	
	\item  	In Section \ref{section6}, we give explicit examples of symplectic left Leibniz algebras. In particular, for an example of four dimensional symplectic Lie algebra $\G$, we give many examples of six dimensional symplectic left Leibniz algebra built upon $\G$.  Finally, the appendix contains the details of the computation used in the proof of our main theorem.

\end{enumerate}

All algebras in this article are finite-dimensional on a commutative field $\K$ of characteristic zero.

\section{Symplectic Leibniz algebras}\label{section2}

\subsection{Definitions and basic properties}

Let $(\A,\bullet)$ be an algebra. For any $u\in\A$, denote by  $\Li_u$ and $\Ri_u$ the left and the right multiplication, i.e., $\Li_uv=u\bullet v$ and $\Ri_uv=v\bullet u$. Put
\begin{equation}\label{sp} u\bullet v=[u,v]+u\diamond v 
\end{equation}
where
\[ [u,v]=\frac12\left(u\bullet v-v\bullet u\right)\esp u\diamond v=\frac12\left(u\bullet v+v\bullet u\right). \]

\begin{enumerate}\item $(\A,\bullet)$ is called a {\it left symmetric algebra} if,  for any $u,v,w\in\A$,
	\begin{equation}\label{ls}
		\ass(u,v,w)=\ass(v,u,w) \quad\mbox{or}\quad [\Li_u,\Li_v]=\Li_{(u\bullet v-v\bullet u)},
	\end{equation}where $\ass(u,v,w)=(u\bullet v)\bullet w-u\bullet(v\bullet w)$. It is well-known that a left symmetric algebra is Lie-admissible, i.e., the associated bracket $\br$ is a Lie bracket.

	\item  $(\A,\bullet)$ is called a {\it left Leibniz algebra} if,  for any $u,v,w\in\A$,
	\begin{equation}\label{l1}
		u\bullet(v\bullet w)=(u\bullet v)\bullet w+v\bullet (u\bullet w)\quad\mbox{or}\quad \mathrm{L}_{u\bullet v}=[\mathrm{L}_u,\mathrm{L}_v].
	\end{equation}
	It is called  a {\it right Leibniz algebra} if,  for any $u,v,w\in\A$,
	\begin{equation}\label{r1}
		(v\bullet w)\bullet u=(v\bullet u)\bullet w+v\bullet (w\bullet u)\quad\mbox{or}\quad \mathrm{R}_{w\bullet u }=-[\mathrm{R}_w,\mathrm{R}_u].
	\end{equation}
	If $(\A,\bullet)$ is both left and right Leibniz then it is called a {\it symmetric Leibniz algebra}.

	It is obvious that  $(\A,\bullet)$ is a left Leibniz algebra if and only if $(\A,\bullet_{\mathrm{opp}})$ is a right Leibniz algebra when $u\bullet_{\mathrm{opp}} v=v\bullet u$. It is obvious also that any Lie algebra is a symmetric Leibniz algebra. However, the class of Leibniz algebras is far more large than the class of Lie algebras and many results of the theory of Lie algebras can be extended to left Leibniz algebras \cite{nil}.
\end{enumerate}

Let $(\A,\bullet)$ be a	 left (resp. right) Leibniz algebra. We denote 
$$\Lei(\A)=\mathrm{span}\left\{ u\bullet v+v\bullet u,\;u,v\in\A  \right\}.$$ 
It is easy to check that $\Lei(\A)$ is an ideal\footnote{In a Leibniz algebra, an ideal is a left and right ideal.} of $\A$, for any $u\in\Lei(\A)$, $\mathrm{L}_u=0$ (resp $\Ri_u=0$), $\Lie(\A)=\A/\Lei(\A)$ is a Lie algebra and $\A$ is a Lie algebra if and only if $\Lei(\A)=\{0\}$. We call $\Lei(\A)$ the {\it Leibniz ideal} of $\A$.

The following proposition is the starting point of this work. 

\begin{pr}\label{first} Let $(\A,\bullet)$ be an algebra and $\om$ is nondegenerate bilinear skew-symmetric form such that
	\[ u\star v-v\star u=[u,v] \]where $u\star v=-\Li_u^*v$ and $\Li_u^*$ is the adjoint of $\Li_u$ with respect to $\om$. Then $(A,\bullet)$ is a left Leibniz algebra if and only if $\star$ is left symmetric and  $\Li_{u\diamond v}=0$ for any $u,v\in\A$.

\end{pr}

\begin{proof} We denote by $\mathrm{K}$ the left multiplication of $\star$. We have, for any $u,v\in\A$,
	\begin{align*}
		\Li_{u\bullet v}-[\Li_u,\Li_v]
		=(\Li_{u\bullet v}^*)^*+[\Li_u^*,\Li_v^*]^*
		&=(\Li_{u\star v-v\star u}^*)^*+(\Li_{u\diamond v}^*)^*+[\Li_u^*,\Li_v^*]^*\\
		&=-\mathrm{K}_{[u,v]_\star}^*+[\mathrm{K}_u^*,\mathrm{K}_v^*]^*+(\Li_{u\diamond v}^*)^*\\
		&=([\mathrm{K}_u^*,\mathrm{K}_v^*]-\mathrm{K}_{[u,v]_\star})^*+(\Li_{u\diamond v}^*)^*.
	\end{align*}
	This completes the proof.
	\end{proof}

A similar computation gives the following proposition.

\begin{pr}\label{firstbis} Let $(\A,\bullet)$ be an algebra and $\om$ is nondegenerate skew-symmetric form such that
	\[ u\star v-v\star u=-[u,v] \]where $u\star v=-\Ri_u^*v$ and $\Ri_u^*$ is the adjoint of $\Ri_u$ with respect to $\om$. Then $(A,\bullet)$ is a right Leibniz algebra if and only if $\star$ is left symmetric and  $\Ri_{u\diamond v}=0$ for any $u,v\in\A$.

\end{pr}

Let us introduce the main notions of this paper.

\begin{Def} \begin{enumerate}
		\item A symplectic form $\om$ on a left Leibniz $(\A,\bullet)$  is a nondegenerate skew-symmetric bilinear form $\om$ satisfying
		\begin{equation}\label{syml}
			\Li_v^*u-\Li_u^*v=[u,v] \end{equation}
		for any $u,v\in \A$ where $\Li_u^*$ is the adjoint of $\Li_u$ with respect to $\om$. 
		We call $(\A,\bullet,\om)$ a symplectic left Leibniz algebra.
		
		\item  A symplectic form $\om$ on a right Leibniz algebra $(\A,\bullet)$  is a nondegenerate skew-symmetric form $\om$ satisfying
		\begin{equation}\label{symr}
			\Ri_v^*u-\Ri_u^*v=-[u,v] \end{equation}
		for any $u,v\in \A$ where $\Ri_u^*$ is the adjoint of $\Ri_u$ with respect to $\om$. We call $(\A,\bullet,\om)$ a symplectic right Leibniz algebra.
		
	\end{enumerate}
\end{Def}

It is evident that symplectic Leibniz algebras generalizes symplectic Lie algebras and
as an immediate consequence of Propositions \ref{first} and \ref{firstbis} we can see that, similarly to symplectic Lie algebras, Leibniz algebras carry a left symmetric product.

\begin{pr} Let $(\A,\bullet,\om)$ be a symplectic left (resp. right) Leibniz algebra. Then $(\A,\star_l)$ (resp. $(\A,\star_r)$) is a left symmetric algebra where
	\begin{equation}\label{star} u\star_l v=-\Li_u^*v\esp u\star_r v=-\Ri_u^*v,\quad u,v\in\A. 
	\end{equation}In particular, $(\A,\bullet)$ is Lie-admissible.
	
\end{pr}

The relations \eqref{sp}, \eqref{syml} and \eqref{symr} give us the following useful characterizations.

\begin{pr}\label{car} \begin{enumerate}
		\item Let $(\A,\bullet)$ be a left Leibniz algebra and $\om$ a nondegenerate bilinear skew-symmetric form on $\A$. Then $\om$ is symplectic if and only if, for any $u,v,w\in\A$, one of the following equivalent relations holds:
		\begin{enumerate}
			\item[$(\mathrm{l}1)$] $\om(u,v\bullet w)-\om(v,u\bullet w)=\frac12\om(u\bullet v,w)-\frac12\om(v\bullet u,w)$.
			\item[$(\mathrm{l}2)$] $d\om(u,v,w)=\om(v,u\diamond w)-\om(u,v\diamond w)$ where
			\[ d\om(u,v,w)=\om(u,[v,w])+\om(v,[w,u])+\om(w,[u,v]). \]
		\end{enumerate}
		\item Let $(\A,\bullet)$ be a right Leibniz algebra and $\om$ a nondegenerate bilinear skew-symmetric form on $\A$. Then $\om$ is symplectic if and only if, for any $u,v,w\in\A$, one of the following equivalent relations holds:
		\begin{enumerate}
			\item[$(\mathrm{r}1)$] $\om(u,w\bullet v)-\om(v,w\bullet u)=\frac12\om(v\bullet u,w)-\frac12\om(u\bullet v,w)$.
			\item[$(\mathrm{r}2)$] $d\om(u,v,w)=\om(u,v\diamond w)-\om(v,u\diamond w).$ 
		\end{enumerate}
	\end{enumerate}
	
\end{pr}

This proposition shows once more that the notions of symplectic left Leibniz algebras and symplectic right Leibniz algebras are a natural generalization of the notion of symplectic Lie algebras.

\begin{exem} \begin{enumerate}
		\item In dimension 2 the only non Lie symplectic left (right) Leibniz algebra  is the symmetric Leibniz algebra given by
		\[ e_2\bullet e_2=e_1\esp \om=\al e_1^*\wedge e_2^*. \]
		
		\item Consider the left Leibniz algebra $(\R^4,\bullet)$ where the non vanishing  Leibniz products are given by
		\[ e_1\bullet e_1=e_4,\; e_1\bullet e_2=e_3,\; e_1\bullet e_3=e_4,\; e_2\bullet e_1=-e_3\esp e_3\bullet e_1=-e_4.  \]
		Finding the bilinear skew-symmetric forms $\om$ satisfying the relation $(\mathrm{l}1)$ in Proposition \ref{car} consists of solving a linear system and it is an easy task. It turns out that there are nondegenerate solutions. For instance, the form $\om=e_1^*\wedge e_4^*+e_2^*\wedge e_3^*$ is a solution and hence it defines a symplectic form on $(\R^4,\bullet)$. Thus we solved a linear system and we get a left symmetric product $\star$  given by
		\[ e_1\star e_1=-e_2+e_3,\; e_1\star e_2=e_3,\; e_2\star e_2=-e_4,\; e_3\star e_1=-e_4. \]

	\end{enumerate}
	
\end{exem}

\subsection{The core of left Leibniz symplectic algebras}

The notions of symplectic left and right Leibniz algebras are equivalent so, trough this paper, we restrict our  study to  symplectic left Leibniz algebras. For any symplectic left Leibniz algebra $(\A,\bullet,\om)$, we denote by $\star$ the associated left symmetric product given by $u\star v=-\Li_u^*v$ and satisfying
\begin{equation}\label{st}
	u\star v-v\star u=[u,v]:=\frac12(u\bullet v-v\bullet u).
\end{equation}
For any vector space $I\subset\A$, we denote by $I^\perp$ its orthogonal with respect to $\om$.

\begin{pr}\label{fati} Let $(\A,\bullet,\om)$ be symplectic left Leibniz algebra. Suppose that there exists a nondegenerate ideal $I$ such that $\Lei(\A)\subset I$ and, for any $u\in I$, $\Li_u=0$. Then $(\A,\bullet,\om)$ is a symplectic Lie algebra. In particular, if $\A$ is non Lie then $\Lei(A)$ is degenerate.
	
\end{pr}

\begin{proof} From  Proposition \ref{car} $(\mathrm{l}1)$, if $\Li_u=0$ then, for any $v,w\in\A$,
	\[ \om(u,v\bullet w)=-\frac12\om(v\bullet u,w). \]
	Now, let $u,w\in I$ and $v\in\A$.
	\begin{align*}
		\om(u,v\bullet w)&=-\frac12\om(v\bullet u,w)	
		=\frac12\om(w,v\bullet u)
		=-\frac14\om(v\bullet w,u).
	\end{align*}
	Thus $\om(u,v\bullet w)=0$ which implies that $\A\bullet I\subset I^\perp\cap I$. Hence, for any $u\in I$, $\Ri_u=0$. Let us show that $I^\perp$ is a Lie algebra. Indeed, for any
	 $u,v\in I^\perp$ and $w\in I$, by virtue of \eqref{st},
	  $$u\star w=u\bullet w+w\star u-\frac12(u\bullet w+w\bullet u).$$ Since
	$u\bullet w-\frac12(u\bullet w+w\bullet u)\in I$, we have
	$$\om(u\bullet v,w)=-\om(v,u\star w)=-\om(v,w\star u)=\om(w\bullet v,u)=0$$ and hence $u\bullet v\in I^\perp$ and $u\bullet v+v\bullet u\in I\cap I^\perp=\{0\}$. Finally, $u\bullet v=-v\bullet u$ and so $I^\perp$ is a Lie algebra. 
	
	On the other hand, $\A=I\oplus I^\perp$ and since $I^\perp$ is a Lie algebra and for any $u\in I$, $\Li_u=\Ri_u=0$, we deduce that for any $u,v\in\A$, $u\bullet v+v\bullet u=0$ which shows that $\A$ is a Lie algebra.\end{proof}

The following proposition is a crucial step in our investigation of symplectic left Leibniz algebras.

\begin{pr}\label{core} Let $(\A,\bullet,\om)$ be a symplectic left  Leibniz algebra. Then the following assertions hold.
	\begin{enumerate}
		\item[$(i)$] $\Lei(\A)$ is an  ideal with respect to both $\bullet$ and $\star$.
		\item[$(ii)$] $I=\Lei(\A)\cap \Lei(\A)^\perp$ is  an isotropic  ideal for both $\bullet$ and $\star$,  $I^\perp$ is an  ideal for both $\bullet$ and $\star$, $\A\bullet\A\subset I^\perp$, $\A\star\A\subset I^\perp$ and $\A\bullet I=\A\star I=0$.   
		\item[$(iii)$] The quotient $\G= I^\perp/I$ has a structure of  symplectic Lie algebra. We denote by $\om_\G$ and $\br_\G$ the corresponding symplectic form  and  Lie bracket.   
		\item[$(iv)$] The quotient of $(\A,\star)$ by the ideal $I^\perp$ is a trivial algebra, i.e., for any $u,v\in\A$, $\pi(u\star v)=0$ where $\pi: \A\too\h=\A/I^\perp$ is the canonical projection. 
	\end{enumerate}
	We call $(\G,\br_\G,\om_\G)$ and $\h$ the core of $(\A,\bullet,\om)$.
\end{pr}

\begin{proof} 
	\begin{enumerate}
		\item[$(i)$] It is known that $\Lei(\A)$ is an ideal for $\bullet$. For any $u\in\Lei(\A)$ and $v\in \A$. We have $\Li_u=0$ and hence $u\star v=0$. On the other hand, according to \eqref{sp} and \eqref{syml},
		\[ v\bullet u=v\star u+u\diamond v. \]But $v\bullet u,u\diamond v\in\Lei(\A)$ and hence $v\star u\in\Lei(\A)$. This shows that $\Lei(\A)$ is an ideal for $\star$.
		\item[$(ii)$] Let $u\in\A$ and $v\in \Lei(\A)^\perp$. For any $w\in\Lei(\A)$,
		\[ \om(u\star v,w )=-\om(v,u\bullet w)=0\esp \om(u\bullet v,w)=-\om(v,u\star w)=0. \]
		This shows that $\Lei(\A)^\perp$ is a left ideal for both $\star$ and $\bullet$. Hence $I$ and $I^\perp=\Lei(\A)+\Lei(\A)^\perp$ are left ideals for both $\star$ and $\bullet$. Moreover, $I\bullet\A=I\star\A=0$ thus $I$ is an ideal for both $\bullet$ and $\star$.

		On the other  hand, for any $u\in I^\perp$, $v\in\A$ and $w\in I$, since $I\subset \Lei(\A)^\perp$, we have
		\[ \om(u\bullet v+v\bullet u,w)=0, \]but $I^\perp$ is a left ideal so $\om(v\bullet u,w)=0$ and we get that $\om(u\bullet v,w)=0$ which shows that
		$I^\perp$ is a right ideal for $\bullet$. Now, for any $u\in\A$, $v\in I^\perp$,
		\[ v\star u=v\bullet u+u\star v-v\diamond u\in I^\perp. \] 
		Thus $I^\perp$ is a right ideal for $\star$.
		
		Now, for any $u,w\in \A$ and $v\in I$, according to Proposition \ref{car},
		\begin{align*}
			0&=\om(u,v\bullet w)-\om(v,u\bullet w)-\frac12\om(u\bullet v,w)+\frac12\om(v\bullet u,w)\\
			&=-\om(v,u\bullet w)+\frac12\om(v,u\star w).
		\end{align*}
		Thus $u\bullet w-\frac12u\star w\in I^\perp$. Also, $u\bullet w+w\bullet u\in\Lei(\A)\subset I^\perp$ and $u\star w-w\star u=\frac12(u\bullet w-w\bullet u)$.
		Thus, if $\pi:\A\too \A/I^\perp$ then 
		\[ \pi(u\star w)-\pi(w\star u)=\pi(u\bullet w)=\frac12\pi(u\star w). \]
		and we get
		\[ \pi(w\star u)=\frac12\pi(u\star w) \]for any $u,w\in\A$. This is equivalent to $\pi(u\star w)=0$ for any $u,v\in\A$. Moreover,
		\[ \pi(u\bullet w)=\frac12\pi(u\bullet w-w\bullet u)+\frac12\pi(u\bullet w+w\bullet u)=0. \]So far we have shown that $\A\bullet\A\subset I^\perp$ and $\A\star\A\subset I^\perp$. Finally, for any $u,w\in\A$ and $v\in I$,
		\[ \om(u\bullet v,w)=-\om(v,u\star w)=0\esp \om(u\star v,w)=-\om(v,u\bullet w)=0. \]Thus $\A\bullet I=0$ and $\A\star I=0$.
		
		\item[$(iii)$]  It is obvious that $\G$ carries both a left Leibniz product $\bullet_\G$ and a nondegenerate skew-symmetric bilinear form $\om_\G$. By using the relation $(\mathrm{l}1)$ in Proposition \ref{car}, one can see that $\om_\G$ is a symplectic form.
		Moreover,
		\[ \G=\pi_\G(\Lei(\A))\oplus \pi_\G(\Lei(\A)^\perp) \]
	 where $\pi_\G: \A\too\G=I^\perp/I$ is the canonical projection. This shows that $\pi_\G(\Lei(\A))$ is a nondegenerate ideal, contained in $\{u\in\G,\Li_u=0\}$,  $\Lei(\G)\subset \pi_\G(\Lei(\A))$ and the result follows from Proposition \ref{fati}.
		
		\item[$(iv)$] It is an immediate consequence of $(ii)$.
		\qedhere
	\end{enumerate}
\end{proof}

\section{Bi-symplectic symmetric Leibniz algebras}
\label{section3}

In this section, we give a complete description of symmetric Leibniz algebras endowed with a  skew-symmetric bilinear form which is symplectic for both the right and the left Leibniz structure.

Let $(\A,\bullet)$ be a symmetric Leibniz algebra and $\om\in\wedge^2\A^*$. We call $\om$ bi-symplectic if $(\A,\bullet,\om)$ is a symplectic left Leibniz algebra as well as a symplectic right Leibniz algebra. We call $(\A,\bullet,\om)$ a bi-symplectic symmetric Leibniz algebra.

\begin{pr}\label{symplecticsym}
	Let $(\A,\bullet)$ be a symmetric Leibniz algebra and $\om\in\wedge^2\A^*$.  Then $\om$ is bi-symplectic   if and only if, for any $u,v,w\in\A$,
	$$\om([u,v],w)+\om([v,w],u)+\om([w,u],v)=0\esp \om(u\diamond w,v)=\om(v\diamond w,u).$$

\end{pr}

\begin{proof}It is an immediate consequence of Proposition \ref{car} $(\mathrm{l}2)$ and $(\mathrm{r}2)$.
\end{proof}

On the other hand, according to \cite[Proposition 2.11]{saidbar}, any symmetric Leibniz algebra $(\A,\bullet)$ is given by 
\[ u\bullet v=[u,v]+\rho(u,v), \]where $\br$ is a Lie bracket on $\A$ and $\rho$ is a bilinear symmetric map  
$\rho \, : \A \times \A \longrightarrow Z(\A)$ ($Z(\A)$ is the center of $(\A,\br)$) such that, for any $u,\, v,\, w \in \A$,
\begin{equation}\label{eq}
	\rho([u, v], w) = \rho(\rho(u, v), w) = 0.
\end{equation}
We call $(\A,\br+\rho)$ the symmetric Leibniz algebra obtained from the Lie algebra $(\A,\br)$ by mean of $\rho$. 

By using this characterization of symmetric Leibniz algebras and  Proposition \ref{symplecticsym}, we get a  complete description of bi-symplectic symmetric Leibniz algebras. But before, we  give a complete description of bi-symplectic symmetric Leibniz algebras whose underlying Lie algebra is abelian.

\begin{theo} \label{commutative} Let $(\A,\bullet,\om)$
	be a bi-symplectic 	  commutative symmetric Leibniz algebra. Then there exists a symplectic vector space $(B,\om_B)$, a vector space $\h$ and a symmetric 3-linear form $T:\h\times\h\times\h\too\K$ such that $(\A,\bullet,\om)$ is isomorphic to $(\h\oplus B\oplus \h^*,\bullet,\om_n)$ where 
	\[ \begin{cases}
	\Li_a=\Li_\al=0,\;	\h\bullet\h\subset\h^*,\;\prec X\bullet Y,Z\succ=T(X,Y,Z),\\
		\om_n (X+a+\al,Y+b+\be)=\prec\al,Y\succ-\prec\be,X\succ+\om_B(a,b),\\
		X,Y,Z\in\h, a,b\in B, \al,\be\in\h^*.
	\end{cases} \]
	Moreover, the associated left symmetric product is given by $u\star v=-u\bullet v$.

\end{theo}
\begin{proof} 
	From the description of symmetric Leibniz algebras aforementioned, a commutative symmetric Leibniz algebra is a commutative algebra $(\A,\bullet)$ such that for any $u,v\in\A$, $\Li_u\circ\Li_v=0$. Put $I=\sum_{u\in\A}\mathrm{Im}\Li_u$.
	
	 Suppose that $\om$ is a bi-symplectic form on $\A$. By virtue of Proposition \ref{symplecticsym}, for any $u\in\A$, $\Li_u$ is skew-symmetric with respect to $\om$,  
	 $$I^\perp=\bigcap_{u\in\A}\ker\Li_u\esp I\subset I^\perp. $$ Choose a complement $B$ of $I$ in $I^\perp$, i.e., 
	 $I^\perp=I\oplus B$.  The restriction $\om_B$ of $\om$  to $B$ is nondegenerate,   there exists a totally isotropic vector space $\h$ such that $B^\perp=I\oplus\h$ and $\A=I\oplus B\oplus \h$. The symplectic form identifies $I$ to the dual $\h^*$ and we can identify $(\A,\om)$ to $(\h^*\oplus B\oplus\h,\om_n)$.
	With this identification in mind, the bracket $\bullet$ satisfies, $\Li_u=0$ for any $u\in \h^*\oplus B$,   $\h\bullet\h\subset \h^*$ and the 3-linear form given by
	\[ T(X,Y,Z):=\prec X\bullet Y,Z\succ=\prec X\bullet Z,Y\succ, \]for any $X,Y,Z\in\h$ is symmetric and defines $\bullet$.
\end{proof}

We give now a complete description of bi-symplectic symmetric Leibniz algebras.

\begin{theo}\label{co} Let $(\A,\bullet)$ be a  symmetric Leibniz algebra obtained from $(\A,\br)$ by mean of $\rho$,  $\om\in\wedge^2\A^*$ and $Z(\A)$ the center of $(\A,\br)$.  Then $\om$ is bi-symplectic if and only if  $(\A,\br,\om)$ is a symplectic Lie algebra and there exists a 3-linear symmetric form $T:\A\times\A\times\A\too\K$ and a totally isotropic vector subspace $I\subset Z(\A)$ such that $T(I^\perp,\cdot,\cdot)=0$ and 
	\begin{equation*}\label{T} \om(\rho(u,v),w) =T(u,v,w),\quad u,v,w\in\A. 
	\end{equation*}
\end{theo}

\begin{proof} According to Proposition \ref{symplecticsym} and the description of symmetric Leibniz algebras aforementioned, if $\om$ is bi-symplectic
	then $(\A,\br,\om)$ is a symplectic Lie algebra,  $\rho$ satisfies $I=\mathrm{span}\{\rho(u,v),u,v\in\A\}\subset Z(\A)$, 
	\[ \rho([u,v],w)=\rho(\rho(u,v),w)=0\esp \om( \rho(u,v),w)=\om (\rho(u,w),v),\quad u,v,w\in\A.  \]
	For any $u,v,w,z\in\G$,
	\[  \om( \rho(\rho(u,v),w),z) =\om( \rho(w,z),\rho(u,v)) =0\]and hence $I$ is totally isotropic. 
	Put $T(u,v,
	w)= \om( \rho(u,v),w)$. Then $T$ is symmetric and, for any $w\in I^\perp$ and $u,v\in\A$,
	\[ T(u,v,w)=\om(\rho(u,v),w) =0. \]
	
	Conversely, let $(\A,\br,\om)$ be a symplectic Lie algebra, $I\subset Z(\A)$ a totally isotropic vector subspace and $T:\A\times\A\times\A\too\K$ a
	3-linear symmetric form such that $T(I^\perp,\cdot,\cdot)=0$. Define the symmetric bilinear application $\rho:\A\times\A\too\A$ by
	\[ T(u,v,w)=\om(\rho(u,v),w) ,\;u,v,w\in\A. \]It is obvious that $\rho$ takes its values in $I$ and, for any $u,v,w,z\in\A$,
	\[\begin{cases}
		\om( \rho(\rho(u,v),w),z)=T(\rho(u,v),w,z)=0,\\
		\om( \rho([u,v],w),z)=\om(\rho(w,z),[u,v])=
		\om([\rho(w,z),u],v)+\om([v,\rho(w,z)],u)=0.
	\end{cases} \]
	So $(\A,\br+\rho,\om)$ is a bi-symplectic symmetric Leibniz algebra.
\end{proof}

This result shows that bi-symplectic symmetric Leibniz algebras are obtained quite easily from symplectic Lie algebras.

In dimension 4, according to the classification of Ovando \cite{ovando1}, there are five symplectic Lie algebras with a non trivial center, namely, $\mathfrak{rh}_3$, $\mathfrak{rr}_{3,0}$, $\mathfrak{rr}_{3,-1}$, $\mathfrak{r}_{4,0}$ and $\mathfrak{n}_{4}$. For each of them we take an isotropic subspace in the center and we determine $\rho$ by using Theorem 
\ref{co} and we get the following result.

\begin{theo}
	\begin{enumerate}
		\item Let $(\A,\bullet,\om)$ be a bi-symplectic non-Lie Leibniz algebra of dimension 2. Then $(\A,\bullet,\om)$ is isomorphic to $(\R^2,\circ,\om_0)$ where the non vanishing brackets and $\om_0$ are given by
		\[ e_2\circ e_2= xe_1\esp \om_0=e^{12},\; x\not=0. \]
		
		\item Let $(\A,\bullet,\om)$ be a bi-symplectic non-Lie Leibniz algebra of dimension 4. Then $(\A,\bullet,\om)$ is isomorphic to one of the following models:
		\begin{enumerate}
			\item there exists a basis $(e_1,e_2,e_3,e_4)$ where the non vanishing brackets and $\om_0$ are given by
			\[ e_1\circ e_1=x e_3+ye_4,\;e_1\circ e_2=e_2\circ e_1=ye_3+ze_4,\;
		e_2\circ e_2=ze_3+te_4,\;\om_0=e^{13}+e^{24},\;x,y,z,t\in\K. \]
		
		\item there exists a basis $(e_1,e_2,e_3,e_4)$  where the non vanishing brackets and $\om_0$ are given by
		\[ e_1\circ e_1=xe_4,\;\om_0=e^{14}+e^{23},\;x\not=0. \]

			\item there exists a basis $(e_1,e_2,e_3,e_4)$ where the non vanishing brackets and $\om_0$ are given by
			\[
			\begin{cases}e_1\circ e_2=(1+z)e_3+ye_4,\;
			e_2\circ e_1=(z-1)e_3+ye_4, e_1\circ e_1=ye_3+xe_4,\;\\ e_2\circ e_2=te_3+ze_4,\; \om_0=e^{14}+e^{23},\;x,y,z,t\in\K.\end{cases} \]
			
			\item there exists a basis $(e_1,e_2,e_3,e_4)$ where the non vanishing brackets and $\om_0$ are given by
			\[e_1\circ e_2=-e_2\circ e_1=e_3,\;
			e_2\circ e_2=xe_3,\; \om_0=e^{14}+e^{23},\;x\not=0. \]

			\item there exists a basis $(e_1,e_2,e_3,e_4)$ where the non vanishing brackets and $\om_0$ are given by
			\[\begin{cases}
			e_1\circ e_2=(1+\frac{x}a)e_3+xe_4,\; e_2\circ e_1=(\frac{x}a-1)e_3+xe_4,\;
			e_1\circ e_1=xe_3+axe_4,\\\;e_2\circ e_2=\frac{x}{a^2}e_3+\frac{x}ae_4,\; \om_0=e^{14}+e^{23},\;x\not=0.\end{cases} \]

			\item there exists a basis $(e_1,e_2,e_3,e_4)$ where the non vanishing brackets and $\om_0$ are given by
			\[e_1\circ e_2=-e_2\circ e_1=e_3,\;
			e_1\circ e_1=xe_4,\; \om_0=e^{14}+e^{23},\;x\not=0. \]

			\item there exists a basis $(e_1,e_2,e_3,e_4)$ where the non vanishing brackets and $\om_0$ are given by
			\[\begin{cases}e_1\circ e_2=(1+x)e_3+\frac{x}ae_4,\; e_2\circ e_1=(x-1)e_3+\frac{x}ae_4,\;
			e_1\circ e_1=\frac{x}{a}e_3+\frac{x}{a^2}e_4,\\\;e_2\circ e_2=axe_3+xe_4,\; \om_0=e^{14}+e^{23},\;x\not=0.\end{cases} \]

			\item there exists a basis $(e_1,e_2,e_3,e_4)$ where the non vanishing brackets and $\om_0$ are given by
			\[e_1\circ e_2=-e_2\circ e_1=e_2,\; e_4\circ e_4=xe_3,\; \om_0=e^{12}+e^{34},\;x\not=0. \]

			\item there exists a basis $(e_1,e_2,e_3,e_4)$ where the non vanishing brackets and $\om_0$ are given by
			\[\begin{cases}e_1\circ e_2=-e_2\circ e_1=e_2,\; e_3\circ e_3=xe_3-axe_4,\;e_3\circ e_4=e_4\circ e_3=\frac{x}ae_3-xe_4,\\\;e_4\circ e_4=\frac{x}{a^2}e_3-\frac{x}{a}e_4,\; \om_0=e^{12}+e^{34},\;x\not=0.\end{cases} \]

			\item there exists a basis $(e_1,e_2,e_3,e_4)$ where the non vanishing brackets and $\om_0$ are given by
			\[e_1\circ e_2=-e_2\circ e_1=e_2,\; e_3\circ e_3=xe_4,\; \om_0=e^{12}+e^{34},\;x\not=0. \]

			\item there exists a basis $(e_1,e_2,e_3,e_4)$ where the non vanishing brackets and $\om_0$ are given by
			\[\begin{cases}e_1\circ e_2=-e_2\circ e_1=e_2,\; e_3\circ e_3=xe_3-\frac{x}ae_4,\;e_3\circ e_4=e_4\circ e_3=axe_3-xe_4,\;\\e_4\circ e_4=a^2xe_3-axe_4,\; \om_0=e^{12}+e^{34},\;x\not=0.\end{cases} \]

			\item there exists a basis $(e_1,e_2,e_3,e_4)$ where the non vanishing brackets and $\om_0$ are given by
			\[e_1\circ e_2=-e_2\circ e_1=e_2,\;e_1\circ e_3=-e_3\circ e_1=-e_3,\; e_1\circ e_1=xe_4,\; \om_0=e^{14}+e^{23},\; x\not=0. \]
			
			\item there exists a basis $(e_1,e_2,e_3,e_4)$ where the non vanishing brackets and $\om_0$ are given by
			\[ e_4\circ e_1=-e_1\circ e_4=e_1,\; e_4\circ e_3=-e_3\circ e_4=e_2,\; e_3\circ e_3=xe_2,\; \om_0=e^{14}\pm e^{23},\; x\not=0. \]
			
			\item there exists a basis $(e_1,e_2,e_3,e_4)$ where the non vanishing brackets and $\om_0$ are given by
			\[ e_4\circ e_1=-e_1\circ e_4=e_2,\;e_4\circ e_2=-e_2\circ e_4=e_3,\;e_4\circ e_4=xe_3,\; \om_0=e^{12}+e^{34},\; x\not=0. \]

		\end{enumerate}

	\end{enumerate}
\end{theo}

\section{ Description of symplectic left Leibniz algebras }\label{section4}
The purpose of this section is to rebuild a symplectic left Leibniz algebra from its core defined in Proposition \ref{core} by using a method inspired from the process of double extension introduced by Medina-Revoy and $T^*$-extension introduced by Bordemann (see \cite{Medina, bor}). We use this description to derive several significant results. 

Let $(\A,\bullet,\om)$ be a symplectic left Leibniz algebra, $\h$ and $(\G,\br_\G,\om_\G)$  its core defined in Proposition \ref{core}. According to this proposition:\begin{enumerate}
	\item $I=\Lei(\A)\cap \Lei(\A)^\perp$ is a totally isotropic  ideal and, for any $u\in I$, $\mathrm{L}_u=\Ri_{u}=0$.
	\item $I^\perp=\Lei(\G)+\Lei(\G)^\perp$ is an  ideal and $\A\bullet\A\subset I^\perp$.
\end{enumerate}
Choose a vector subspace $B$ such that $I^\perp=B\oplus I$. There exists a totally isotropic vector subspace $\h_0$ such that $B^\perp=I\oplus \h_0$. The restriction $\om_B$ of $\om$ to $B$ is nondegenerate and $\flat:I\too\h_0^*$, $u\mapsto \om_B(u,.)$ is an isomorphism. So  $(B^\perp,\om_B)$ is isometric $\h_0\oplus\h_0^*$ endowed with the symplectic from $\om_0$ given by
\[ \om_0(u+\al,v+\be)=\prec\al,v\succ-\prec\be,u\succ. \] 
and hence $(\A, \om)$ is isomorphic to $(\h_0\oplus \g\oplus \h_0,\om_n=\om_0+\om_B)$. Under this identification, the Leibniz product $\bullet$ is identified with the product $\circ$ given by
\begin{equation*} \begin{cases}X\circ Y=\theta(X,Y)+\theta_2(X,Y),\;
		a\circ  b=a\bullet_Bb+\mu(a,b),\\
		X\circ a=F(X)a+\nu(X,a),\;
		a\circ X=G(X)a+\rho(a,X),\; \\
		\Li_\al=\Ri_\al=0,\; X,Y\in \h_0, a,b\in B,  \al\in\h_0^*,
	\end{cases} 
\end{equation*}
 where    
\[\theta(X,Y)\in B,\;\theta_2(X,Y)\in\h_0^*,\; a\bullet_B b\in B,\; \mu(a,b),\nu(X,a),\rho(a,X)\in\h_0^*, \esp F(X),G(X)\in\mathrm{End}(B).\]  

Let $\pi_\h:\A\too\h=\A/I^\perp$ and $\pi_\G:I^\perp\too \G=I^\perp/I$. It is clear that $\pi_\h$ identifies $\h_0$ with $\h$ and $\pi_\G$ identifies $(B,\bullet_B,\om_B)$ with $(\G,\br_\G,\om_\G)$. Define $\Om:\h\too\otimes^2\h^*$ and $\psi,\xi:\h\times\h\too\G$ by the relations 
$$\Om(X)(Y,Z)=\prec\theta_2(X,Y),Z\succ, \prec\rho(a,X),Y\succ=\om_\G(\xi(X,Y),a)\esp
\prec\nu(X,a),Y\succ=\om_\G(\psi(X,Y),a). $$
On the other hand, since $\om_n$ is symplectic, by virtue of Proposition \ref{car}, for any $X\in\h$, $a,b\in\G$,
\begin{align*}	0&=\om_n(X,a\circ b)-\om_n(a,X\circ  b)-\frac12\om_n(X\circ a,b)+\frac12\om_n(a\circ X,b)\\
	&=-\prec\mu(a,b),X\succ+\om_\G(F(X)b,a)-\frac12
	\om_\G(F(X)a,b)+\frac12\om_\G(G(X)a,b).	
\end{align*}So $\mu(a,b)=\om_\G(K(\;\centerdot\;)a,b)$ where $K(X)=\frac12G(X)-\frac12F(X)-F(X)^*$. 

To summarize, we have showed that $(\A, \bullet, \om )$ is isomorphic to $(\h\oplus \g\oplus \h^*, \circ, \om_n)$, $\circ$ and $\om_n$ are given by 
\begin{equation}\label{dd} 
	\begin{cases}X\circ Y=\theta(X,Y)+\Om(X)(Y,
		\;\centerdot\;),\;
		a\circ  b=[a, b]_\G+\om_\G(K(\;\centerdot\;)a,b),\;\\
		X\circ a=F(X)a+\om_\G(\psi(X,\;\centerdot\;),a),\;
		
		a\circ X=G(X)a+\om_\G(\xi(X,\;\centerdot\;),a),\; \\
		\Li_\al=\Ri_\al=0,\\
		\om_n (X+a+\al,Y+b+\be)=\prec\al,Y\succ-\prec\be,X\succ+\om_\G(a,b),\\
		X,Y\in \h, a,b\in \G,  \al,\be\in\h^*,
	\end{cases} 
\end{equation} 
$F,G:\h\too \mathrm{End}(\G)$, $K(X)=\frac12G(X)-\frac12F(X)-F(X)^*$, $\Om:\h\too \otimes^2\h^*$,   
$\theta,\xi,\psi:\h\times\h\too \G$  are bilinear maps. The quantity $\Om(X)(Y,\;\centerdot\;)$ is the element of $\h^*$ which maps $Z$ to $\Om(X)(Y,Z)$ and the quantities $\om_\G(K(\;\centerdot\;)a,b),\om_\G(\psi(X,\;\centerdot\;),a),\om_\G(\xi(X,\;\centerdot\;),a)$ are defined in a similar way.

A straightforward computation (see Appendix~\ref{app} for the details) shows that $(\h\oplus \g\oplus \h^*, \circ, \om_n)$ is a symplectic left Leibniz algebra if and only if $F$ and $G$ take their values in the Lie algebra of derivations of $\G$, say $\mathrm{Der}(\G)$, and the data $(F,G,\xi,\psi,\theta,\Om)$ satisfy the following system
\begin{eqnarray}\label{eqmain3} \begin{cases}
		\Om(X)(Z,Y)-\Om(Y)(Z,X)=\frac12\Om(X)(Y,Z)-\frac12\Om(Y)(X,Z),\\
		\psi(X,Y)-\psi(Y,X)
		=\frac12(\theta(X,Y)-\theta(Y,X)),\\
		\theta(X,Y)=\xi(Y,X)+\frac12\psi(X,Y)-\frac12\xi(X,Y),\\
		\om_\G(\theta(X,Y),\xi(Z,T))-	\om_\G(\theta(Y,Z),\psi(X,T))
		+\om_\G(\theta(X,Z),\psi(Y,T))=0,\\
		F(X)\theta(Y,Z)-F(Y)\theta(X,Z)-G(Z)\theta(X,Y)=0
		,\\
		F(X)^*\psi(Y,Z)-F(Y)^*\psi(X,Z)+K(Z)\theta(X,Y)=0,\\
		F(X)^*\xi(Y,Z)-G(Y)^*\psi(X,Z)
		-K(Z)^*\theta(X,Y)=0,\\
		(F(X)+G(X))\xi(Y,Z)=0,\\
		F(X)^*+G(X)^*=-(F(X)+G(X)),\\
		K(Y)(F(X)+G(X))=0,\\
		G(Y)(F(X)+G(X))=0,\\
		\Ri_{\psi(X,Y)}^{\star_\G}+K(Y)F(X)+F(X)^*K(Y)=0,\\
		\Ri_{\xi(X,Y)}^{\star_\G}+K(Y)^*G(X)
		+G(X)^*K(Y)=0,\\
		\ad^{\G}_{\theta(X,Y)}=[F(X),F(Y)]=-[F(X),G(Y)]
		,\\
		\ad_{(F(X)+G(X))a}^\G=0,\\
		K(X)([a, b]_\G)=a\star_\G K(X)b-b\star_\G K(X)a,
\end{cases}\end{eqnarray} where $\star_\G$ is the left symmetric product associated to $(\G,\be_\G,\om_\G)$, $\Ri^{\star_\G}$ is the associated right multiplication, $\ad^\G$ is the adjoint representation of $(\G,\br_\G)$ and $F(X)^*$ is the adjoint with respect to $\om_\G$.

It turns out that the system~\eqref{eqmain3} can be reduced drastically. For that, we establish some preliminary results. 

Let $(\G,\br_\G,\om_\G)$ be a symplectic Lie algebra, $Z(\G)$ its center,  $\h$ a vector space
and $(F,G,\theta,\psi,\xi, \Om)$	satisfying $\mathrm{Im}F,\mathrm{Im}G\subset\mathrm{Der}(\G)$ and \eqref{eqmain3}. Put $S(X)=F(X)+G(X)$ for any $X\in \h$. 
Then $$K(X)=\frac12S(X)-F(X)-F(X)^*\esp  K(X)-K(X)^*=S(X).$$

\begin{pr}\label{pr} With the notations an hypothesis above, we have, for any $X,Y\in\h$,
	\[\begin{cases}\theta(X,Y)+\theta(Y,X)= \psi(X,Y)+\xi(X,Y)\in Z(\G),\\ F(X)\theta(Y,Z)-F(Y)\theta(X,Z)+F(Z)\theta(X,Y)\in Z(\G),\\
	S(X)S(Y)=F(X)S(Y)=S(Y)F(X)=0.	
	\end{cases} \]
\end{pr}

\begin{proof} From the equation 14 in \eqref{eqmain3}, we deduce that $\theta(X,X)\in Z(\G)$ for any $X\in\G$ and  from the equations 2 and 3 in \eqref{eqmain3}, we have
	\begin{align*}
		2\left(\psi(X,Y)-\psi(Y,X)\right)&=	\theta(X,Y)-\theta(Y,X)\\&=\xi(Y,X)-\xi(X,Y)+\frac12\psi(X,Y)-\frac12\psi(Y,X)-\frac12\xi(X,Y)+\frac12\xi(Y,X)\\
		&=-\frac32\left(\xi(X,Y)-\xi(Y,X)\right)+\frac12\left(\psi(X,Y)-\psi(Y,X)\right).
	\end{align*}
	So
	\begin{equation}\label{psi} \psi(X,Y)-\psi(Y,X)=\xi(Y,X)-\xi(X,Y).  \end{equation}
	Moreover, 
	\begin{align*}
		\theta(X,Y)+\theta(Y,X)&=\xi(Y,X)+\xi(X,Y)+\frac12\psi(X,Y)+\frac12\psi(Y,X)-\frac12\xi(X,Y)-\frac12\xi(Y,X)\\
		&=\frac12(\psi(X,Y)+\psi(X,Y))+\frac12(\xi(X,Y)+\xi(Y,X)).
	\end{align*}
	Thus
	\[ \psi(X,Y)+\psi(X,Y)=-(\xi(X,Y)+\xi(Y,X))+2(\theta(X,Y)+\theta(Y,X)). \]
	By adding this relation and \eqref{psi} we get
	\[ \psi(X,Y)+\xi(X,Y)=\theta(X,Y)+\theta(Y,X) \]and the first relation follows. To show the second relation, put
	\[ Q=F(X)\theta(Y,Z)-F(Y)\theta(X,Z)+F(Z)\theta(X,Y). \]Having in mind that $F$ takes its values in $\mathrm{Der}(\G)$ and the relation $\ad_{\theta(X,Y)}^\G=[F(X),F(Y)]$, we have
	\begin{align*}
		\ad_Q^\G&=[F(X),\ad_{\theta(Y,Z)}^\G]
		-[F(Y),\ad_{\theta(X,Z)}^\G]+[F(Z),\ad_{\theta(X,Y)}^\G]\\
		&=\oint_{X,Y,Z}[F(X),[F(Y),F(Z)]]=0
	\end{align*}and we get the second relation. 
Finally, by using the equations 10 and 11 in~\eqref{eqmain3}, we get 
\[ \frac12S(Y)S(X)-F(Y)S(X)-F^*(Y)S(X)=(S(Y)-F(Y))S(X)=0\]and hence
\begin{align*} 
-\frac12S(Y)S(X)-F(Y)^*S(X)=0. \end{align*}
We have $S(X)^*=-S(X)$ and, from the equation 14, $[F(X),F(Y)]=-[F(X),S(Y)-F(Y)]=0$ and hence $[F(X),S(Y)]=0$. So, by taking the adjoint and the equation 14 ,we get
\[ S(X)F(Y)=\frac12S(X)S(Y)=F(Y)S(X)=S(Y)S(X). \]
So $S(X)S(Y)=2S(Y)S(X)=4S(X)S(Y)$.  Thus $ S(X)S(Y)=F(X)S(Y)=S(Y)F(X)=0$.
\end{proof}

\begin{Le}\label{le} Let $(\G,\br,\om)$ be a symplectic Lie algebra, $\star$ the associated left symmetric product and $D$ a derivation of $\G$. Then, for any $a,b\in\G$,
	\[ (D+D^*)([a,b])=a\star(D+D^*)b-b\star(D+D^*)a. \]
\end{Le}
\begin{proof} Recall that the left multiplication associated to $\star$ is given by $\Li=-\ad^*$. Since $D$ is a derivation, we have  $[D,\ad_a]=\ad_{D(a)}$ for any $a\in \G$, and hence $[\Li_a,D^*]=\Li_{D(a)}$. Thus, for any $a,b \in \G$, 
	\[D^*(a\star b)= a\star D^*b-D(a)\star b \]
	and
	\begin{align*}
		D^*(a\star b)&= a\star D^*b-D(a)\star b,\\
		(D+D^*)([a,b])&=[Da,b]+[a,Db]+a\star D^*b-D(a)\star b-b\star D^*a+D(b)\star a,\\
		&=-b\star Da+a\star Db+a\star_\g D^*b-b\star D^*a,\\
		&=a\star (D+D^*)b-b\star(D+D^*)a
	\end{align*}and we get the desired relation.
\end{proof}
Let us show now that the last equation in \eqref{eqmain3} is redundant. Indeed, since $S(X)^*=-S(X)$ the relation $\ad^\g_{S(X)a}=0$ is equivalent to $S(X)(a\star_\g b)=0$. The fact $S(X)a\in Z(\G)$ implies that $\Ri_{S(X)a}^{\star_\G}=0$ and the last equation in \eqref{eqmain3} is equivalent to
\[ (F(X)+F(X)^*)([a,b]_\G)=a\star_\G (F(X)+F(X)^*)b-b\star_\G(F(X)+F(X)^*)a \]
which is always valid since $F(X)$ is a derivation (see Lemma \ref{le}).

 To summarize, we  proved  the following  theorem.

\begin{theo}\label{main} Let $(\A,\bullet,\om)$ be a symplectic left Leibniz algebra,  $(\G,\br_\G,\om_\G)$ and $\h$ its core. Then $(\A,\bullet,\om)$ is isomorphic to $(\h\oplus\G\oplus\h^*,\circ,\om_n)$ where
	$(\circ, \om_n)$ are given by~\eqref{dd} 
	and the data $(F, G,\xi, \psi, \theta, \Om)$ satisfy $F,G:\h\too \mathrm{Der}(\G)$, $S(X)=F(X)+G(X)$, $K(X)=\frac12S(X)-F(X)-F(X)^*$ and, for any $X,Y,Z,T\in\h$, $a,b\in \G$,
	\begin{eqnarray}\label{eqmain3mo} \begin{cases}
			\Om(X)(Z,Y)-\Om(Y)(Z,X)=\frac12\Om(X)(Y,Z)-\frac12\Om(Y)(X,Z),\\
			\psi(X,Y)-\psi(Y,X)
			=\xi(Y,X)-\xi(X,Y),\\
			\theta(X,Y)=\xi(Y,X)+\frac12\psi(X,Y)-\frac12\xi(X,Y),\\
			\om_\G(\theta(X,Y),\xi(Z,T))-	\om_\G(\theta(Y,Z),\psi(X,T))
			+\om_\G(\theta(X,Z),\psi(Y,T))=0,\\
			F(X)\theta(Y,Z)-F(Y)\theta(X,Z)+F(Z)\theta(X,Y)=S(Z)\theta(X,Y)
			,\\
			F(X)^*\psi(Y,Z)-F(Y)^*\psi(X,Z)+K(Z)\theta(X,Y)=0,\\
			F(X)^*(\psi(Y,Z)+\xi(Y,Z))+S(Y)\psi(X,Z)
			+S(Z)\theta(X,Y)=0,\\
			\ad^{\G}_{\theta(X,Y)}=[F(X),F(Y)],\\
			\Ri_{\psi(X,Y)}^{\star_\G}=
			(F(Y)+F(Y)^*)F(X)+F(X)^*(F(Y)+F(Y)^*),\\
			\Ri_{\psi(X,Y)+\xi(X,Y)}^{\star_\G}=0,
			\\S(X)(a\star_\G b)=0,
			S(X)^*=-S(X),
			S(X)\xi(Y,Z)=0,
			S(X)S(Y)=F(X)S(Y)=S(X)F(Y)=0
			.
	\end{cases}\end{eqnarray}
	Conversely, if $(\G,\br_\G,\om_\G)$ is a symplectic Lie algebra, $\h$ is a trivial algebra and $(F,G,\xi,\psi, \theta, \Om)$ a list as above satisfying~\eqref{eqmain3mo} then $(\h\oplus \G\oplus \h^*,\circ,\om_n)$, where $(\circ, \om_n)$ are given by \eqref{dd}, is a symplectic left Leibniz algebra. We call it the symplectic left symplectic algebra obtained from $(\G,\br_\G,\om_\G)$ and $\h$ by means of $(F,G,\xi,\psi,\theta,\Om)$.\end{theo}

To complete  Theorem \ref{main}, we give the left symmetric product associated to the product \eqref{dd}.

\begin{co} The left symmetric product associated to the product \eqref{dd} is given by
	\begin{equation}\label{lf} \begin{cases}
			X\star Y=\psi(X,Y)+\Om(X)(\;\centerdot\;,Y),\;
			X\star a=-F(X)^*a+\om_\G(\theta(X,\;\centerdot\;),a),\\
			a\star X=K(X)a+\om_\G(\xi(\;\centerdot\;,X),a),\;
			a\star b=a\star_\G b+\om_\G(G(\;\centerdot\;)a,b),\\
			\Li_\al^\star=\Ri_\al^\star=0.
	\end{cases} \end{equation}

\end{co}

\section{Some subclasses of symplectic left Leibniz algebras}\label{section5}

In this section, we derive from Theorem \ref{main} the description of some natural subclasses of symplectic left Leibniz algebras.
The following result is an immediate consequence of Theorem~\ref{main}.
\begin{theo}\label{lagrangian} Let $(\A,\bullet,\om)$ be a symplectic left Leibniz algebra such that $\Lei(\A)$ is Lagrangian. Then $\h=\A/\Lei(\A)$ and $\G=\{0\}$ and $(\A,\bullet,\om)$ is isomorphic to $(\h\oplus\h^*,\circ,\om_n)$ where $\circ$ and the associated left symmetric product are given by
	\begin{equation}\label{cohom} (X+\al)\circ(Y+\be)=\Om(X)(Y,\;\centerdot\;),\quad
		(X+\al)\star (Y+\be)=\Om(X)(\;\centerdot\;,Y) \end{equation}
	and $\Om:\h\too\h^*\otimes\h^*$ satisfying
	\[ 	\Om(X)(Z,Y)-\Om(Y)(Z,X)=\frac12\Om(X)(Y,Z)-\frac12\Om(Y)(X,Z), \]for any $X,Y,Z\in\h$. 
\end{theo}

\begin{theo}\label{isotropic} Let $(\A,\bullet,\om)$ be a symplectic left Leibniz algebra such that its core $(\G,\br_\G,\om_\G)$ satisfies  $Z(\G)=\{0\}$.  Then $(\A,\bullet,\om)$ is isomorphic to $(\h\oplus\G\oplus\h^*,\circ,\om_n)$, where
	$(\circ, \om_n)$  are given by \eqref{dd}, the data $(F,G,\xi,\psi,\theta,\Om)$ satisfy $G=-F$, $K=-F-F^*$, $\theta$ is skew-symmetric, $\xi=-\psi$ and
	\begin{eqnarray}\label{eqmain3bis} \begin{cases}
			\Om(X)(Z,Y)-\Om(Y)(Z,X)=\frac12\Om(X)(Y,Z)-\frac12\Om(Y)(X,Z),\\
			\theta(X,Y)=\psi(X,Y)-\psi(Y,X),\\
			\om_\G(\theta(X,Y),\psi(Z,T))+	\om_\G(\theta(Y,Z),\psi(X,T))
			+\om_\G(\theta(Z,X),\psi(Y,T))=0,\\
			F(X)^*\psi(Y,Z)-F(Y)^*\psi(X,Z)+K(Z)\theta(X,Y)=0,\\
			\Ri_{\psi(X,Y)}^{\star_\G}+K(Y)F(X)+F(X)^*K(Y)=0,\\
			\ad^{\G}_{\theta(X,Y)}=[F(X),F(Y)].
	\end{cases}\end{eqnarray}
\end{theo}
\begin{proof} It is a direct consequence of Theorem \ref{main} and the assumption $Z(\G)=\{0\}$. According to Proposition \ref{pr}, $\theta$ is skew-symmetric, $\psi=-\xi$, $F=-G$, $K=-F-F^*$, the equation 5 in \eqref{eqmain3mo} is satisfied due to Proposition \ref{pr}.
\end{proof}

There is a particular case of Theorem \ref{isotropic} where the system \eqref{eqmain3bis} can be reduced drastically.

Suppose that we are in the situation of Theorem \eqref{eqmain3bis} and moreover any derivation of $\G$ is inner. Since $Z(\G)=\{0\}$ there exists a unique  endomorphism $H:\h\too\G$ such that $F(X)=\ad_{H(X)}$, for any $X\in\G$, and from the last equation in \eqref{eqmain3bis} we deduce that $\theta(X,Y)=[H(X),H(Y)]_\G$, for any $X,Y\in\h$. Put $\psi(X,Y)=H(X)\star_\G H(Y)+\psi_1(X,Y)$. Then the second equation in \eqref{eqmain3bis} is equivalent to $\psi_1$ is symmetric. Note that $K(X)=-\ad_{H(X)}+\Li_{H(X)}^{\star_\G}=\Ri_{H(X)}^{\star_\G}$ and the equation 5 of \eqref{eqmain3bis} can be written
\[ \Ri_{\psi_1(X,Y)}^{\star_\G}+
\Ri_{H(X)\star_\G H(Y)}^{\star_\G}-
\Ri_{H(Y)}^{\star_\G}\circ \Ri_{H(X)}^{\star_\G}
+[\Ri_{H(Y)}^{\star_\G},\Li_{H(X)}^{\star_\G}]=0. \]But $\star_\G$ is left symmetric and this equation is equivalent to $\Ri_{\psi_1(X,Y)}^{\star_\G}=0$. Let us check that  the equations 3 and 4 of \eqref{eqmain3bis} are also satisfied. Indeed, we have
\begin{align*} \om_\G(\theta(X,Y),\psi(Z,T))&=\om_\G([H(X),H(Y)]_\G,H(Z)\star_\G H(T))+\om_\G([H(X),H(Y)]_\G,\psi_1(Z,T))\\
&=-\om_\G([H(Z),[H(X),H(Y)]_\G]_\G,H(T))-\om_\G(H(Y),H(X)\star_\G \psi_1(Z,T)).
 \end{align*}So the equation 3 of \eqref{eqmain3bis} holds as a consequence of the Jacobi identity and $\Ri_{\psi_1(X,Y)}^{\star_\G}=0$. On the other hand, if $Q=F(X)^*\psi(Y,Z)-F(Y)^*\psi(X,Z)+K(Z)\theta(X,Y)$ then
\begin{align*}
	Q&=-H(X)\star_\G(H(Y)\star_\G H(Z))-H(X)\star_\G\psi_1(Y,Z)+H(Y)\star_\G(H(X)\star_\G H(Z))+H(Y)\star_\G\psi_1(X,Z)\\&+
	[H(X),H(Y)]_\G\star_\G H(Z)=0
\end{align*}as a consequence of the fact that $\star_\G$ is left symmetric and $\Ri_{\psi_1(X,Y)}^{\star_\G}=0$.

We proved that
 $(\A,\bullet,\om)$ is isomorphic to $(\h\oplus\G\oplus\h^*,\circ,\om_n)$ where
	 \begin{equation*}
	\begin{cases}X\circ Y=[H(X),H(Y)]_\G+\Om(X)(Y,
		\;\centerdot\;),\;
		a\circ  b=[a, b]_\G+\om_\G([a,b]_\G,H(\;\centerdot\;)),\;\\
		X\circ a=[H(X),a]_\G+\om_\G(\psi(X,\;\centerdot\;),a)
		+\om_\G([H(X),a]_\G,H(\;\centerdot\;))
		,\;\\
		
		a\circ X=[a,H(X)]_\G-\om_\G(\psi(X,\;\centerdot\;),a)
		-\om_\G([H(X),a]_\G,H(\;\centerdot\;)),\; \\
		\Li_\al=\Ri_\al=0,\\
		\om_n (X+a+\al,Y+b+\be)=\prec\al,Y\succ-\prec\be,X\succ+\om_\G(a,b),\\
		X,Y\in \h, a,b\in \G,  \al,\be\in\h^*,
	\end{cases} 
\end{equation*} where $\psi$ is symmetric, $\Ri_{\psi(X,Y)}^{\star_\G}=0$, $H:\h\too\G$ is an endomorphism and $\Om$ satisfies
\[ \Om(X)(Z,Y)-\Om(Y)(Z,X)=\frac12\Om(X)(Y,Z)-\frac12\Om(Y)(X,Z). \]

Now, consider the isomorphism $A$ of $\h\oplus\G\oplus\h^*$ given by $A(X+a+\al)=X+a+H(X)+\al$. Then $(\h\oplus\G\oplus\h^*,\circ,\om_n)$ is isomorphic to $(\h\oplus\G\oplus\h^*,\circ_H,\om_H)$ where $\om_H=(A^{-1})^*\om_n$ and $u\circ_H v=A(A^{-1}u\circ A^{-1}v)$. We have
\begin{align*}
	X\circ_H Y&=A((X-H(X))\circ(Y-H(Y))\\
	&=A([H(X),H(Y)]_\G-[H(X),H(Y)]_\G-
	\om_\G(\psi(X,\;\centerdot\;),H(Y))-\om_\G([H(X),H(Y)]_\G,H(\;\centerdot\;))
\\	&-[H(X),H(Y)]_\G+\om_\G(\psi(Y,\;\centerdot\;),H(X))+\om_\G([H(X),H(Y)]_\G,H(\;\centerdot\;)))
	+[H(X),H(Y)]_\G\\&+\Om(X)(Y,\;\centerdot\;)-\om_\G([H(X),H(Y)]_\G,H(\;\centerdot\;)))\\
	&=\Om(X)(Y,\;\centerdot\;)-\om_\G([H(X),H(Y)]_\G,H(\;\centerdot\;))-
	\om_\G(\psi(X,\;\centerdot\;),H(Y))+
	\om_\G(\psi(Y,\;\centerdot\;),H(X)),\\
	a\circ_H b&=a\circ b,\\
	X\circ_H a&=A((X-H(X))\circ a))\\
	&=[H(X),a]_\G-[H(X),a]_\G+\om_\G([H(X),a]_\G,H(\;\centerdot\;))+\om_\G(\psi(X,\;\centerdot\;),a)-\om_\G([H(X),a]_\G,H(\;\centerdot\;))\\
	&=\om_\G(\psi(X,\;\centerdot\;),a),\;\\
	a\circ_H X&=A(a\circ (X-H(X))\\
	&=[a,H(X)]_\G-[a,H(X)]_\G+\om_\G([H(X),a]_\G,H(\;\centerdot\;))-\om_\G(\psi(X,\;\centerdot\;),a)-\om_\G([H(X),a]_\G,H(\;\centerdot\;))\\
	&=-\om_\G(\psi(X,\;\centerdot\;),a),\\
	\om_H(X+a+\al,Y+b+\be)&=\prec\al,Y\succ-\prec\be,X\succ+\om_\G(a-H(X),b-H(Y)).
\end{align*} 
Put
\[ \Om_1(X)(Y,\;\centerdot\;)=-\om_\G([H(X),H(Y)]_\G,H(\;\centerdot\;))-
\om_\G(\psi(X,\;\centerdot\;),H(Y))+
\om_\G(\psi(Y,\;\centerdot\;),H(X)). \]

\begin{align*}
	\Om_1(X)(Z,Y)&=-\om_\G([H(X),H(Z)]_\G,H(Y))-
	\om_\G(\psi(X,Y),H(Z))+
	\om_\G(\psi(Z,Y),H(X)),\\
	-\Om_1(Y)(Z,X)&=\om_\G([H(Y),H(Z)]_\G,H(X))+
	\om_\G(\psi(Y,X),H(Z))-
	\om_\G(\psi(Z,X),H(Y)),\\
	\frac12\Om_1(X)(Y,Z)&=-\frac12\om_\G([H(X),H(Y)]_\G,H(Z))-\frac12
	\om_\G(\psi(X,Z),H(Y))+\frac12
	\om_\G(\psi(Y,Z),H(X)),\\
	-\frac12\Om_1(Y)(X,Z)&=\frac12\om_\G([H(Y),H(X)]_\G,H(Z))+\frac12
	\om_\G(\psi(Y,Z),H(X))-\frac12
	\om_\G(\psi(X,Z),H(Y))
\end{align*}and we can see that $\Om_1$ satisfies
\[  \Om_1(X)(Z,Y)-\Om_1(Y)(Z,X)=\frac12\Om_1(X)(Y,Z)-\frac12\Om_1(Y)(X,Z). 
\]

Finally, we proved the following result.
\begin{theo}\label{isotropicbis} Let $(\A,\bullet,\om)$ be a symplectic left Leibniz algebra such that its core $(\G,\br_\G,\om_\G)$ satisfies  $Z(\G)=\{0\}$ and any derivation is inner.  Then $(\A,\bullet,\om)$ is isomorphic to $(\h\oplus\G\oplus\h^*,\circ,\om_H)$, where
	$(\circ, \om_H)$  are given by 
	\[ \begin{cases}
		X\circ Y=\Om(X)(Y,\;\centerdot\;),\;
		[a,b]=[a,b]_\G-\om_\G([a,b]_\G,H(\;\centerdot\;)),\\
		X\circ a=-a\circ X=\om_\G(\psi(X,\;\centerdot\;),a),\;
		\Li_\al=\Ri_\al=0,\\
		\om_H(X+a+\al,Y+b+\be)=\prec\al,Y\succ-\prec\be,X\succ+\om_\G(a-H(X),b-H(Y)),\\
		X,Y\in\h,a,b\in\G,\al,\be\in\h^*
	\end{cases} \]where $H:\h\too\G$ an endomorphism, $\psi:\h\times\h\too\G$ is symmetric, $\Ri_{\psi(X,Y)}^{\star_\G}=0$ and $\Om:\h\too\otimes^2\h^*$ satisfies
\[ \Om(X)(Z,Y)-\Om(Y)(Z,X)=\frac12\Om(X)(Y,Z)-\frac12\Om(Y)(X,Z). \]

\end{theo}

There is one more interesting situation which can be derived from Theorem \ref{main}.

	\begin{theo}\label{h=1} Let $(\A,\bullet,\om)$ be a symplectic Leibniz algebra such that $\dim(\Lei(\A)\cap\Lei(\A)^\perp)=1$. Then $(\A,\bullet,\om)$ is isomorphic to $(\K e\oplus\G\oplus \K e^*,\circ,\om_n)$ where $(\G,\br_\G,\om_\G)$ is a symplectic Lie algebra and $\circ$ is given by
		\begin{equation}\label{ddh=1} \begin{cases}{e}\circ {e}=\frac12(a_0+b_0)+\la e^*,\;
				a\circ  b=[a, b]_\G+\om_\G(Ka,b) e^*,\;\\
				e\circ a=Fa+\om_\G(a_0,a)e^*,\;
				a\circ e=-Fa+Sa+\om_\G(b_0,a)e^*,\; \\
				e\circ e^*=e^*\circ e=e^*\circ a=a\circ e^*=0,
			\end{cases} 
		\end{equation}where 
		$F\in \mathrm{Der}(\G)$, $K=\frac12S-F-F^*$, $a_0,b_0,\in\G$, $c_0=\frac12(a_0+b_0)$, $\la\in\K$ and
		\begin{eqnarray}\label{eqmain3mobis} \begin{cases}
				\om_\G(a_0,b_0)=0,Sa_0=0,Sb_0=0,Fc_0=F^*c_0=0\\
				\ad^{\G}_{c_0}=\Ri_{c_0}^{\star_\G}=0,\\
				\Ri_{a_0}^{\star_\G}=
				(F+F^*)F+F^*(F+F^*),
				\\S(a\star_\G b)=0,
				S^*=-S,
				S^2=FS=SF=0
				.
		\end{cases}\end{eqnarray}
		Moreover, the left symmetric product associated to \eqref{ddh=1} is given by
		\begin{equation} \label{ddh=1lf}\begin{cases}
				e\star e=a_0+\la e^*,\;
				e\star a=-F^*a+\om_\G(c_0,a)e^*,\quad \\
				a\star e=Ka+\om_\G(b_0,a)e^*\quad ,\;
				a\star b=a\star_\G b+\om_\G((S-F)a,b)e^* ,\\
				a\star e^*=e^*\star e=e\star e^*=e^*\star a=e^*\star e^*=0.
		\end{cases} \end{equation}
		
	\end{theo}
	
	\begin{proof}
		According to Theorem~\ref{main}, $(\A, \bullet,\om)$ is isomorphic to $\h\oplus \g\oplus \h^*$ where $(\g, \br_{\g},\om_\g)$ is a symplectic Lie algebra, $\h=\K e$, $ \h^*= \K e^*$ and, the Leibniz product and its associated left symmetric product are given respectively by ~\eqref{dd} and ~\eqref{lf}. Since $\dim \h=1$, one can put \[F(e)=F,\; G(e)=S(e)-F(e)=S-F,\; \psi(e,e)=a_0,\; \xi(e,e)=b_0,\;\Om(e)(e,\;\centerdot\;)=\la e^*\] where $a_0, b_0 \in \g$. 
		The equation 3 in \eqref{eqmain3mo} implies that $\theta(e,e)=\frac12(a_0+b_0)=c_0$. The equations 5,7 and 11 in \eqref{eqmain3mo} implies that
		$Sa_0=Sb_0=0$ and $Fc_0= F^*c_0=0$. Thus, the system \eqref{eqmain3mobis} is obtained. 
		
	\end{proof}
	
	\begin{co} Let $(\A,\bullet,\om)$ be a symplectic left Leibniz algebra such that $\dim(\Lei(\A))=1$. Then $(\A,\bullet,\om)$ is isomorphic to $(\K e\oplus\G\oplus \K e^*,\circ,\om_n)$ where $(\G,\br_\G,\om_\G)$ is a symplectic Lie algebra and $\circ$ is given by
		\begin{equation}\label{ddhbis=1} \begin{cases}{e}\circ {e}=\la e^*,\;
				a\circ  b=[a, b]_\G+\om_\G(Ka,b) e^*,\;\\
				e\circ a=Fa+\om_\G(a_0,a)e^*=-a\circ e,\;
				\al\circ e=e^*\circ a=a\circ e^*=0,
			\end{cases} 
		\end{equation}where 
		$F\in \mathrm{Der}(\G)$, $K=-F-F^*$, $a_0\in\G$, $\la \in \K$ such that $\Ri_{a_0}^{\star_\G}+KF+F^*K=0$.
		
		Moreover, the left symmetric product associated to the product \eqref{ddhbis=1} is given by
		\begin{equation}\label{ddhbislf=1} \begin{cases}
				e\star e=a_0+\la e^*,\;
				e\star a=-F^*a, \\
				a\star e=Ka-\om_\G(a_0,a)e^*\quad ,\;
				a\star b=a\star_\G b-\om_\G(Fa,b)e^* ,\\
				a\star e^*=e^*\star e=e\star e^*=e^*\star a=e^*\star e^*=0.
		\end{cases} \end{equation}

	\end{co}
	\begin{proof} The only thing to check is that $F=-G$ and $\theta$ is skew-symmetric. We have $\dim(\Lei(\A))=1$, then  $\Lei(\A)\subset \Lei(\A)^\perp$ and so $\h^* =\Lei(\A)$. Therefore, for any $X, Y\in \h$ and $a,b\in\G$,
		\[ X\circ Y +Y\circ X\in \h^*  \esp a\circ b+b\circ a\in\h^*, \]
		thus  $F=-G$ ,$\psi=-\xi$ and $\theta$ is skew-symmetric. The other statements are a direct consequence of Theorem \ref{h=1}.\end{proof}

	\section{Examples of symplectic left Leibniz algebras}\label{section6}

	\subsection{Symplectic left Leibniz algebras with 2-dimensional non abelian cores}
	In this subsection, we describe all symplectic left Leibniz algebras when the  the associated symplectic Lie algebra is non abelian and has dimension  2.

	\begin{pr} Let $(\G,\br,\om)$ be the 2-dimensional symplectic non abelian Lie algebra identified to $\R^2$ with its canonical basis $(e,f)$ and $\om=e^*\wedge f^*$ and $[e,f]=\la e$ and let $\h$ be a vector space. Let $(F,\theta,\psi,\xi,\Om)$ a solution of \eqref{eqmain3bis}. Then there exists $\al,\be\in\h^*$ such that
		\[	\begin{cases}
			F(X)=\left(\begin{matrix}
				\prec \al,X\succ&\prec \be,X\succ\\0&0
			\end{matrix} \right),\; K(X)=-\prec \al,X\succ I_2,\\
			\theta=\frac1\la (\al\wedge \be)e,\;\psi=-\xi=
			(\frac1{2\la} \al\wedge \be+\mu)e+\frac1\la (\al\otimes \al) f, \mu(X,Y)=\mu(Y,X)
		\end{cases} \]and $\Om$ satisfies
	\[ \Om(X)(Z,Y)-\Om(Y)(Z,X)=\frac12\Om(X)(Y,Z)-\frac12\Om(Y)(X,Z). \]
	\end{pr}

	\begin{proof} The left symmetric product associated to $(\G,\br,\om)$ is given by 
		$\Ri_{e}=0$ and  $\Ri_{f}=\lambda \mathrm{I}_2$.
		Put
		\[ \theta(X,Y)=\theta_1(X,Y)e+\theta_2(X,Y)f\esp \psi(X,Y)=\psi_1(X,Y)e+\psi_2(X,Y)f. \]It is easy to show that any derivation of $\G$ has the form
		$D=\left(\begin{matrix}
			a&b\\0&0
		\end{matrix} \right)$ and hence, for any $X\in\h$, $$F(X)=\left(\begin{matrix}
		\prec \al,X\succ&\prec \be,X\succ\\0&0
	\end{matrix} \right),\; F(X)^*=\left(\begin{array}{cc}
	0 & -\prec \be,X\succ  
	\\
	0 & \prec \al,X\succ  
\end{array}\right)
 \esp K(X)=-\prec \al,X\succ\mathrm{I}_2$$ where $\al,\be\in\h^*$
	
		 The last equation in ~\eqref{eqmain3bis} is equivalent to 
		\[ \left(\begin{matrix}
			-\la\theta_2(X,Y)&\la\theta_1(X,Y)\\0&0
		\end{matrix} \right)=
		\left(\begin{matrix}
			0&\prec \al,X\succ\prec \be,Y\succ-
			\prec \al,Y\succ\prec \be,X\succ\\0&0
		\end{matrix} \right).
		\]
		So
		$\theta_2(X,Y)=0$ and $\theta_1=\frac1\la(\al\wedge \be).$
		The second equation in ~\eqref{eqmain3bis} is equivalent to, \[ \theta_1(X,Y)=\psi_1(X,Y)-\psi_1(Y,X)\esp \psi_2(X,Y)=\psi_2(Y,X).  \]
		We have
		\[\Ri_{\psi(X,Y)}^*=\la \psi_2(X,Y)I_2, \esp K(Y)F(X)+F(X)^*K(Y)=-\prec \al,X\succ\prec \al,Y\succ I_2\]
		and the equation 5 in ~\eqref{eqmain3bis} is equivalent to $\psi_2=\frac1\lambda\al\otimes\al$.
		The equation 3 in ~\eqref{eqmain3bis} is equivalent to
		\[ (\theta_1\otimes\psi_2)(X,Y,Z,T)+
		(\theta_1\otimes\psi_2)(Y,Z,X,T)+
		(\theta_1\otimes\psi_2)(Z,X,Y,T)=0. \]
		But
		\begin{align*}
			\la^2	\theta_1\otimes\psi_2(X,Y,Z,T)&=\left(
			\prec \al,X\succ\prec \be,Y\succ\prec \al,Z\succ-\prec \al,Y\succ\prec \be,X\succ\prec \al,Z\succ \right)\prec \al,T\succ,\\
			\la^2	\theta_1\otimes\psi_2(Y,Z,X,T)&=\left(
			\prec \al,Y\succ\prec \be,Z\succ\prec \al,X\succ-\prec \al,Z\succ\prec \be,Y\succ\prec \al,X\succ \right)\prec \al,T\succ,\\
			\la^2	\theta_1\otimes\psi_2(Z,X,Y,T)&=\left(
			\prec \al,Z\succ\prec \be,X\succ\prec \al,Y\succ-\prec \al,X\succ\prec \be,Z\succ\prec \al,Y\succ \right)\prec \al,T\succ,
		\end{align*}
		which shows that equation 3 in ~\eqref{eqmain3bis} is satisfied. Put
		\[Q= F(X)^*\psi(Y,Z)-F(Y)^*\psi(X,Z)+K(Z)\theta(X,Y). \]We have
		\begin{align*}
			Q&=\frac1\la\prec\al,Y\succ\prec\al,Z\succ\left(-\prec \be,X\succ e+\prec \al,X\succ f \right)-\frac1\la\prec\al,X\succ\prec\al,Z\succ\left(
			-\prec \be,Y\succ e+\prec \al,Y\succ f \right)\\&-\frac1\lambda\prec \al,Z\succ \left(\prec \al,X\succ\prec \be,Y\succ-\prec \al,Y\succ\prec \be,X\succ\right)e=0
		\end{align*}and hence the equation 4 in ~\eqref{eqmain3bis} is satisfied. This completes the proof. 
		\end{proof}
	
	\begin{co} Let $(\A,\bullet,\om)$ be a symplectic Leibniz algebra such that its core $\G$ is 2-dimensional symplectic non abelian Lie algebra. Then $(\A,\bullet,\om)$ is isomorphic to $(\h\oplus\R^2\oplus \h^*,\circ,\om_n)$ where $(\circ, \om_n)$ are given by
		\[ \begin{cases}
			X\circ Y=\frac1\la(\al\wedge \be)(X,Y)e+\Om(X)(Y,\;\centerdot\;),\\
			X\circ e=-e\circ X=\prec \al,X\succ e-\frac1\la \prec \al,X\succ \al,\;\\ X\circ f=-f\circ X=\prec \be,X\succ e+\frac1{2\la}(\al(X)\be-\be(X)\al)+\mu(X,\;\centerdot\;),\\
			e\circ f=-f\circ e=\la e-\al,\\
			e\circ e=f\circ f=0,\\
			\om_n(X+x_1e+x_2f+\al, Y+y_1e+y_2f+\be)=\prec \al, Y \succ-\prec \be, X \succ+ x_1y_2-y_1x_2.
		\end{cases} \]
		where $(e,f)$ is the canonical basis of $\R^2$, $\al, \be \in \h^*$, $\la \neq 0$,  $\mu \in  \h^* \otimes \h^*$ is symmetric and $\Om$ satisfies
		\[ \Om(X)(Z,Y)-\Om(Y)(Z,X)=\frac12\Om(X)(Y,Z)-\frac12\Om(Y)(X,Z) \]for any $X,Y,Z\in\h$.
		
	\end{co}
	
	\subsection{A class of symplectic left Leibniz algebras with 2-dimensional  abelian cores}
	Now, we give two important classes of solutions of \eqref{eqmain3mo} when $(\G,\br,\om)$ is a 2-dimensional abelian Lie algebra. The treatment of the general is tedious and left to the reader. 
	\begin{pr} Let $(\G,\br,\om)$ be a 2-dimensional symplectic  abelian Lie algebra, $\h$ be a vector space and $(F,G,\theta,\psi,\xi,\Om)$ a solution of \eqref{eqmain3mo}.  Then:
		\begin{enumerate}
			\item If there exists $X_0\in\h$ such that $F(X_0)\not=-G(X_0)$ then there exists a basis $(e,f)$ of $\G$, $\al\in\h^*\setminus\{0\}$,  $\be\in\h^*$ such that $\om=e^*\wedge f^*$,  for any $X,Y\in\h$,
			\[ S(X)=\left(\begin{array}{cc}
				0 & \prec \al,X\succ  
				\\
				0 &0  
			\end{array}\right),\; F(X)=\left(\begin{array}{cc}
				0 & \prec \be,X\succ  
				\\
				0 &0  
			\end{array}\right),\;\theta(X,Y),\psi(X,Y),\xi(X,Y)\in\K e\] and $(\theta,\psi,\xi,\Om)$ satisfy
		\[ \begin{cases}
			\Om(X)(Z,Y)-\Om(Y)(Z,X)=\frac12\Om(X)(Y,Z)-\frac12\Om(Y)(X,Z),\\
			\psi(X,Y)-\psi(Y,X)
			=\xi(Y,X)-\xi(X,Y),\\
			\theta(X,Y)=\xi(Y,X)+\frac12\psi(X,Y)-\frac12\xi(X,Y).
		\end{cases} \]
		
		\item If, for any $X\in\h$, $F(X)=-G(X)$ and there exists $X_0\in\h$ such that $\det(F(X_0))\not=0$ then there exists an endomorphism $A$ of $\G$, $\al\in\h^*\setminus\{0\}$ such that $\tr(A)=0$, $F(X)=\prec\al,X\succ A$, $\theta$ skew-symmetric, $\xi=-\phi$, for any $X,Y\in\ker\al$,
		\[ \phi(X,\;\centerdot\;)=0,\;\theta(X,Y)=0,\;\theta(X_0,Y)=\phi(X_0,Y) \]and $\Om$ satisfies
		\[ \Om(X)(Z,Y)-\Om(Y)(Z,X)=\frac12\Om(X)(Y,Z)-\frac12\Om(Y)(X,Z). \]
		
		\end{enumerate}
			\end{pr}
		
		\begin{proof} Suppose that there exists $X_0$ such that $S(X_0)\not=0$. The relation $S(X_0)^2=0$ implies that $\ker S(X_0)=\mathrm{Im}S(X_0)=\K e$. Moreover, for any $X\in\h$, $S(X)S(X_0)=S(X_0)S(X)=S(X)^2=0$ then either $S(X)=0$ or $\ker S(X)=\mathrm{Im}S(X)=\K e$. The relation $F(X)S(X_0)=S(X_0)F(X)=0$ implies that either $F(X)=0$ or $\ker F(X)=\mathrm{Im}F(X)=\K e$.
			 Then there exists a basis $(e,f)$ of $\G$,  $\al\in\h^*\setminus\{0\}$ and $\be\in\h$ such that $\om=e^*\wedge f^*$ and, for any $X\in\G$,
			\[ S(X)=\left(\begin{array}{cc}
				0 & \prec \al,X\succ  
				\\
				0 &0  
			\end{array}\right)\esp F(X)=\left(\begin{array}{cc}
			0 & \prec \be,X\succ  
			\\
			0 &0  
		\end{array}\right). \]Now, $F(X)+F(X)^*=0$ and
	the equation 9 in \eqref{eqmain3mo} is obviously satisfied.
	The relation $S(X_0)\xi(Y,Z)=0$ gives that $\xi(X,Y)=\xi_1(X,Y)e$. Put
	\[ \theta(X,Y)=\theta_1(X,Y)e+\theta_2(X,Y)f
	\esp \psi(X,Y)=\psi_1(X,Y)e+\psi_2(X,Y)f. \]
	From the equation 2 in \eqref{eqmain3mo}, we deduce that $\psi_2$ is symmetric and hence, by virtue of the equation 3 in \eqref{eqmain3mo}, $\theta_2(X,Y)=\frac12\psi_2(X,Y)$. On the other hand, the equation 6 in \eqref{eqmain3mo} gives $\prec\al,X_0\rangle \theta_2(X,X)=0$  hence $\theta_2(X,X)=0$ and $\psi_2=\theta_2=0$. The equations 4-6 are obviously satisfied.

	Suppose now that, for any $X\in\h$, $S(X)=0$ and there exists $X_0\in\h$ such that $\det(F(X_0))\not=0$. A direct computation shows that the relation $$(F(Y)+F(Y)^*)F(X)+F(X)^*(F(Y)+F(Y)^*)=0$$ is equivalent to $\tr(F(X))=0$. So $\{F(X),X\in\h\}$ is an abelian subalgebra of $\mathrm{sl}(2,\K)$ and hence it has dimension 1. So there exists an endomorphism $A$ of $\G$ such that $\tr(A)=0$ and, for any $X\in\h$, $F(X)=\prec \al,X\succ A$ where $\al\in\h^*$. The equations 8-11 in \eqref{eqmain3mo} are now satisfied, let us focus on the others equations. Remark first that $F(X)^*=-F(X)$ and hence $K(X)=0$. Put $\h=\K X_0\oplus\ker\al$. The equation 7 in \eqref{eqmain3mo} gives, for any $X,Y\in\h$, $F(X)^*(\phi(X,Y)+\xi(X,Y))=0$ and hence $\xi=-\phi$. Thus the equation 2 in \eqref{eqmain3mo} is satisfied and the equation 3 is equivalent to $\theta(X,Y)=\psi(X,Y)-\psi(Y,X)$. Now the equation 6 is equivalent to $F(X_0)^*\psi(Y,Z)$ for any $Y\in\ker\al$ and $Z\in\h$ which is equivalent to $\psi(Y,\;\centerdot\;)=0$ for any $Y\in\ker\al$. The equation 5 is equivalent to
	$\theta(X,Y)=0$ for any $X,Y\in\ker\al$. Finally, the equations 5-7 are equivalent to 
	\[ \xi=-\phi,\; \psi(Y,\;\centerdot\;)=0,\;
	\theta(X,Y)=0\esp \theta(X_0,Y)=\psi(X_0,Y) \]for any $X,Y\in\h$. Finally, one can check easily that under the conditions above the equation 4 is satisfied.
		\end{proof}

	As a consequence of this proposition one can build a large class of symplectic left Leibniz algebras.

	\subsection{Examples of six dimensional symplectic left Leibniz algebras}

	Six dimensional symplectic left Leibniz algebras can be determined from   symplectic Lie algebras of dimension 2 or 4. Those obtained from 2 dimensional symplectic Lie algebras were determined in the above subsection. The determination of six dimensional symplectic left Leibniz algebras from four dimensional symplectic  Lie algebras can be undertaken by using Theorem \ref{h=1} and the following steps.
	
	\begin{enumerate}
		\item Since four dimensional symplectic Lie algebras were classified in \cite{ovando1}, we take an element $(\G,\br,\om_\G)$ from the list determined by Ovando in \cite[Proposition 2.4]{ovando1}.
		\item We look for $F,S:\G\too\G$ two derivations and $(a_0,b_0)\in\G\times\G$ satisfying
		\begin{eqnarray}\label{model} \begin{cases}
				\om_\G(a_0,b_0)=0,Sa_0=0,Sb_0=0,Fc_0=F^*c_0=0,\;c_0=\frac12(a_0+b_0),\\
				\ad^{\G}_{c_0}=\Ri_{c_0}^{\star_\G}=0,\\
				\Ri_{a_0}^{\star_\G}=
				(F+F^*)F+F^*(F+F^*),
				\\S(a\star_\G b)=0,
				S^*=-S,
				S^2=FS=SF=0
				.
			\end{cases}
		\end{eqnarray}
		The solution of this system is  straightforward and can be done using a computation software.
	\end{enumerate}
	
	In the following, we will illustrate this method on one example of four dimensional symplectic Lie algebra from \cite[Proposition 2.4]{ovando1}, namely  $\mathrm{rr}_{3,-1}$. 
	
	Recall that on $\mathfrak{rr}_{3,-1}$ the Lie brackets and the symplectic form are given by
	\[ [e_1,e_2]=e_2,\;[e_1,e_3]=-e_3\esp \om_\G=e^{14}+e^{23}. \]
	The solutions $(F,S,c_0,a_0)$ of \eqref{model} are
	\[\begin{cases}F=b_1E_{21}+b_2E_{31}+b_3E_{41}-bE_{22},\; S=sE_{41},

		 \\ c_0=x e_4,\; a_0=z e_1+ b_{1} b e_2+ b_{2} be_3 + ye_4,c_0=\frac12(a_0+b_0),\quad zs=zx=0.\end{cases}
	\]Note that $E_{ij}$ is the matrix where the entry $(i,j)$ is equal to 1 and all the others are zero. We identify $\h$ to $\K e_5$ and $\h^*$ to $\K e_6$.
	The Leibniz products and the symplectic form on $\h\oplus \G\oplus\h^*$ associated to the obtained solution are given by
	$$\begin{cases}
		e_5\circ e_1 = 
		b_{1}e_2 + b_{2}e_3 + b_{3}e_4 
		-ye_6,\;\; e_5\circ e_2 =
		-be_2
		-b_{2} be_6,\; \\ 
		e_5\circ e_3 = 
		be_3+ 
		b_{1} be_6,\;\;
		e_1\circ e_5 =  -b_{1}e_2  -b_{2}e_3+ (s-b_{3} )e_4
		+(y-2 x)e_6, \\
		e_2\circ e_5 =  be_2+
		b_{2} be_6,\;\;
		e_3\circ e_5 =  -b e_3  
		-b_{1} be_6,\;\; e_5\circ e_4=-e_4\circ e_5=ze_6,\\
		e_1\circ e_2 =-e_2\circ e_1= e_2
		+ b_{2}e_6,\;\;
		e_1\circ e_3 =-e_3\circ e_1= -e_3 -b_{1}e_6,\\
		e_1\circ e_1 =  
		-\frac{s}{2}e_6,\;\;	e_5\circ e_5=xe_4+\la e_6,\\
		\om=e^1\wedge e^4+e^2\wedge e^3-e^5\wedge e^6,\\
		zx=zs=0.
	\end{cases}$$
	By replacing $(e_2,e_3)$ by  $(e_2+b_2e_6,e_3+b_1e_6)$ and $y$ by $y+2b_2b_1$ the Leibniz product and $\om$ become$$
	\begin{cases}
		e_5\circ e_1 = 
		b_1e_2 + b_2e_3 + b_3e_4 
		-(y+2b_2b_1)e_6,\;\; e_5\circ e_2 =
		-be_2
		,\; \\ 
		e_5\circ e_3 = 
		be_3,\;\;
		e_1\circ e_5 =  -b_1e_2  -b_2e_3+ (s-b_3 )e_4
		+(y+2b_2b_1-2 x)e_6, \\
		e_2\circ e_5 =  be_2
		,\;\;
		e_3\circ e_5 =  -b e_3,\; e_5\circ e_4=-e_4\circ e_5=ze_6,  
		,\\
		e_1\circ e_2 =-e_2\circ e_1= e_2
		,\;\;
		e_1\circ e_3 =-e_3\circ e_1= -e_3 ,\\
		e_1\circ e_1 =  
		-\frac{s}{2}e_6,\;\;	e_5\circ e_5=xe_4+\la e_6,\\
		\om=e^1\wedge e^4+e^2\wedge e^3-e^5\wedge e^6+b_2e^2\wedge e^5+b_1e^3\wedge e^5,\\
		zx=zs=0.
	\end{cases}$$
	
	If $b=0$, we replace $e_5$ by $e_5+b_1e_2+b_2e_3$ and we get 
	$$\begin{cases}
		e_5\circ e_1 = 
		b_3e_4 
		-(y+2b_2b_1)e_6,\;\; 
		,\; \\ 
		e_1\circ e_5 =  (s-b_3 )e_4
		+(y+2b_2b_1-2 x)e_6,e_5\circ e_4=-e_4\circ e_5=ze_6, \\
		e_1\circ e_2 =-e_2\circ e_1= e_2
		,\;\;
		e_1\circ e_3 =-e_3\circ e_1= -e_3 ,\\
		e_1\circ e_1 =  
		-\frac{s}{2}e_6,\;\;	e_5\circ e_5=xe_4+\la e_6,\\
		\om=e^1\wedge e^4+e^2\wedge e^3-e^5\wedge e^6+2b_2e^2\wedge e^5+2b_1e^3\wedge e^5,\\
		zx=zs=0.
	\end{cases}$$
	
	If $b\not=0$, we replace $e_5$ by $\frac1{b}e_5$ and we get
	$$\begin{cases}
		e_5\circ e_1 = \frac1{b}(
		b_1e_2 + b_2e_3 + b_3e_4 
		-(y+2b_2b_1)e_6),\;\; e_5\circ e_2 =-e_2\circ e_5=
		-e_2
		,\; \\ 
		e_5\circ e_3 = -e_3\circ e_5=
		e_3,\;\;
		e_1\circ e_5 = \frac1{b}( -b_1e_2  -b_2e_3+ (s-b_3 )e_4
		+(y+2b_2b_1-2 x)e_6), \\
		e_5\circ e_4=-e_4\circ e_5=ze_6,\\
		e_1\circ e_2 =-e_2\circ e_1= e_2
		,\;\;
		e_1\circ e_3 =-e_3\circ e_1= -e_3 ,\\
		e_1\circ e_1 =  
		-\frac{s}{2}e_6,\;\;	e_5\circ e_5=
		\frac1{b^2}(xe_4+\la e_6),\\
		\om=e^1\wedge e^4+e^2\wedge e^3-\frac1{b}e^5\wedge e^6+\frac{b_2}{b}e^2\wedge e^5+\frac{b_1}{b}e^3\wedge e^5,\\
		zs=zx=0.
	\end{cases}$$
	We put $\rho=b$ and we multiply all the parameters by $\frac1\rho$ by keeping the same name, we get
	
	$$\begin{cases}
		e_5\circ e_1 = 
		b_1e_2 + b_2e_3 + b_3e_4 
		-(y+2b_2b_1)e_6,\;\; e_5\circ e_2 =-e_2\circ e_5=
		-e_2
		,\; \\ 
		e_5\circ e_3 = -e_3\circ e_5=
		e_3,\;\;
		e_1\circ e_5 =  -b_1e_2  -b_2e_3+ (s-b_3 )e_4
		+(y+2b_2b_1-2 x)e_6, \\
		e_5\circ e_4=-e_4\circ e_5=ze_6,\\
		e_1\circ e_2 =-e_2\circ e_1= e_2
		,\;\;
		e_1\circ e_3 =-e_3\circ e_1= -e_3 ,\\
		e_1\circ e_1 =  
		-\frac{\rho s}{2}e_6,\;\;	e_5\circ e_5=
		\frac1{\rho}(xe_4+\la e_6),\\
		\om=e^1\wedge e^4+e^2\wedge e^3-\frac1{\rho}e^5\wedge e^6+{b_2}e^2\wedge e^5+{b_1}e^3\wedge e^5,\\
		zx=zs=0.
	\end{cases}$$
	We replace $e_1$ by $e_1+b_1e_2-b_2e_3$ and $e_6$ by $\rho e_6$ and we get
	$$\begin{cases}
		e_5\circ e_1 = 
		b_3e_4 
		-(y+2b_2b_1)e_6,\;\; e_5\circ e_2 =-e_2\circ e_5=
		-e_2
		,\; \\ 
		e_5\circ e_3 = -e_3\circ e_5=
		e_3,\;\;
		e_1\circ e_5 =  (s-b_3 )e_4
		+(y+2b_2b_1-2 x)e_6, \\
		e_5\circ e_4=-e_4\circ e_5=ze_6,\\
		e_1\circ e_2 =-e_2\circ e_1= e_2
		,\;\;
		e_1\circ e_3 =-e_3\circ e_1= -e_3 ,\\
		e_1\circ e_1 =  
		-\frac{ s}{2}e_6,\;\;	e_5\circ e_5=
		xe_4+\la e_6,\\
		\om=b_2e^1\wedge e_2+b_1e^1\wedge e^3+
		e^1\wedge e^4+e^2\wedge e^3-e^5\wedge e^6+{b_2}e^2\wedge e^5+{b_1}e^3\wedge e^5,\\
		zx=zs=0.
	\end{cases}$$

	\section{Appendix}\label{app}

	We give the conditions for which  left Leibniz algebras $(\h\oplus \g \oplus\h^*,\circ,\om_n)$ appearing in Theorem \ref{main}  is a symplectic left Leibniz algebra in order to complete the proof of this theorem. 
	We start by sitting the relations for $\om_n$ to be symplectic. Recall from Proposition \ref{car} that $\om_n$ is symplectic if and only if 
	\[ \om_n(u,v\circ w)-\om_n(v,u\circ w)-\frac12\om_n(u\circ v,w)+\frac12\om_n(v\circ u,w)=0, \]for any $u,v,w\in\h\oplus\G\oplus\h^*$.
	Note that this relation holds if $u$, $v$ or $w$ belongs to $\h^*$.

	\begin{enumerate}
		\item For any $X,Y,Z\in \h$,
		\begin{align*}
			0&=\om_n(X,Y\circ Z)-\om_n(Y,X\circ Z)-\frac12\om_n(X\circ Y,Z)+\frac12\om_n(Y\circ X,Z)\\
			&=-\Om(Y)(Z,X)+\Om(X)(Z,Y)-\frac12\Om(X)(Y,Z)+\frac12\Om(Y)(X,Z).
		\end{align*}
		We find the first relation in \eqref{eqmain3}. 
		\item  For $X,Y\in\h$ and $a\in\G$,
		\begin{align*}
			0&=\om_n(X,Y\circ a)-\om_n(Y,X\circ a)-\frac12\om_n(X\circ Y,a)+\frac12\om_n(Y\circ X,a)\\
			&=-\om_\G(\psi(Y,X),a)+\om_\G(\psi(X,Y),a)
			-\frac12\om_\G(\theta(X,Y),a)+\frac12\om_\G(\theta(Y,X),a).
		\end{align*}	
		We find the second relation in \eqref{eqmain3}.
		\item For $X,Y\in\h$ and $a\in\G$,
		\begin{align*}
			0&=\om_n(X,a\circ Y)-\om_n(a,X\circ Y)-\frac12\om_n(X\circ a,Y)+\frac12\om_n(a\circ X,Y)\\
			&=-\om_\G(\xi(Y,X),a)-\om_\G(a,\theta(X,Y))-\frac12\om_\G(\psi(X,Y),a))+\frac12\om_\G(\xi(X,Y),a).
		\end{align*}We find the third relation in \eqref{eqmain3}.

		\item For $X\in\G$ and $a,b\in\G$, 
		\begin{align*}
			0&=\om_n(X,a\circ b)-\om_n(a,X\circ b)-\frac12\om_n(X\circ a,b)+\frac12\om_n(a\circ X,b)\\
			&=-\om_\G(K(X)a,b)-\om_\G(a,F(X)b)-\frac12\om_\G(F(X)a,b)+\frac12\om_\G(G(X)a,b),
		\end{align*}thus $K(X)=\frac12G(X)-\frac12F(X)-F(X)^*$.

		\item For $a,b\in\G$ and $X\in\h$,
		\begin{align*}
			0&=\om_n(a,b\circ X)-\om_n(b,a\circ X)-\frac12\om_n(a\circ b,X)+\frac12\om_n(b\circ a,X)\\
			&=\om_G(a,G(X)b)-\om_\G(b,G(X)a)-\frac12\om_\G(K(X)a,b)+\frac12\om_\G(K(X)b,a).
		\end{align*}So
		\[ G(X)^*+G(X)=\frac12(K(X)+K(X)^*). \] Or
		\[ K(X)+K(X)^*=\frac12G(X)-\frac12F(X)-F(X)^*
		+\frac12G(X)^*-\frac12F(X)^*-F(X). \]
		Thus
		\[ K(X)+K(X)^*=\frac12(G(X)+G(X)^*)-\frac32(F(X)+F(X)^*). \]
		So
		\[ F(X)^*+G(X)^*=-(F(X)+G(X)) \]and we recover the
		the ninth relation in \eqref{eqmain3}.
		\item For $a,b,c\in\G$ the relation is equivalent to   $(\G,\om_\g)$ is a symplectic Lie algebra.
		
	\end{enumerate}
	
	Let us find now the condition for the product $\circ$ given by \eqref{dd} is left Leibniz. Recall that
	\begin{equation*} \begin{cases}X\circ Y=\theta(X,Y)+\Om(X)(Y,
			\;\centerdot\;),\;
			a\circ  b=[a, b]_\G+\om_\G(K(\;\centerdot\;)a,b),\;\\
			X\circ a=F(X)a+\om_\G(\psi(X,\;\centerdot\;),a),\;
			
			a\circ X=G(X)a+\om_\G(\xi(X,\;\centerdot\;),a),\; \\
			\Li_\al^\circ=\Ri_\al^\circ=0,
		\end{cases} 
	\end{equation*}

	This product defines a left Leibniz structure if and only if, for any $u,v,w\in\h\oplus\G\oplus\h^*$,
	\[ Q(u,v,w):=u\circ(v\circ w)-(u\circ v)\circ w-v\circ (u\circ w)=0. \]
	Since, for any $\al\in\h^*$, $\Li_\al=\Ri_\al=0$, we have $Q(\al,.,.)=Q(.,\al,.)=Q(.,.,\al)=0$. So the vanishing of $Q$ is equivalent to
	$$\begin{cases}
		Q(X,Y,Z)=  Q(X,Y,a)= Q(X,a,Y)=  Q(X,a,b)=0,\\
		Q(a,b,c)= Q(a,b, X)= Q(a,X,b)=
		Q(a,X,Y)= 0,
	\end{cases}$$ for any $X,Y,Z\in \h$, $a,b,c\in \G$.

	\begin{enumerate}  
		\item We have 
		\begin{align*}
			X\circ(Y\circ Z)&=X\circ\theta(Y,Z)
			=
			F(X)\theta(Y,Z)+\om_\G(\psi(X,\;\centerdot\;),\theta(Y,Z))\\
			(X\circ Y)\circ Z&=\theta(X,Y)\circ Z
			=G(Z)\theta(X,Y)+\om_\G(\xi(Z,\;\centerdot\;),\theta(X,Y))\\
			Y\circ(X\circ Z)
			&=F(Y)\theta(X,Z)+\om_\G(\psi(Y,\;\centerdot\;),\theta(X,Z)).
		\end{align*}
		$Q(X,Y,Z)=0$ is equivalent to 	  
		\[ \begin{cases}
			F(X)\theta(Y,Z)-F(Y)\theta(X,Z)-G(Z)\theta(X,Y)=0,\\	
			\om_\G(\psi(X,T),\theta(Y,Z))-\om_\G(\xi(Z,T),\theta(X,Y))-\om_\G(\psi(Y,T),\theta(X,Z))=0
		\end{cases} \]
		and we recover the equation 4 and 5 equations of \eqref{eqmain3}.

		\item  
		We have
		\begin{align*}
			X\circ (Y\circ a)&=X\circ  F(Y)a
			= F(X)F(Y)a+\om_\G(\psi(X,\;\centerdot\;),F(Y)a) \\
			(X\circ  Y)\circ a &=  \theta(X,Y)\circ a
			= [\theta(X,Y),a]_\G+\om_\G(K(\;\centerdot\;)\theta(X,Y),a)\\
			Y\circ (X\circ a) &= F(Y)F(X)a+ \om_\G(\psi(Y,\;\centerdot\;),F(X)a).
		\end{align*}
		
		Thus, $Q(X,Y,a)=0$ if and only if
		\[\begin{cases} \ad_{\theta(X,Y)}^\G=[F(X),F(Y)],\\
			F(Y)^*\psi(X,Z)-K(Z)\theta(X,Y)-F(X)^*\psi(Y,Z)=0.
		\end{cases} \]and we recover the equation 6 and the equation 14 in \eqref{eqmain3}.

		\item We have   
		\begin{align*}
			X\circ (a\circ Y)&= X\circ G(Y)a
			= F(X)G(Y)a+\om_\G(\psi(X,\;\centerdot\;),G(Y)a),\\
			(X\circ a)\circ  Y&=F(X)a\circ Y
			=G(Y)F(X)a+\om_\G(\xi(Y,\;\centerdot\;),F(X)a),\\
			a\circ (X\circ Y)&= a\circ\theta(X,Y)
			=[a,\theta(X,Y)]_\G+\om_\G(K(\;\centerdot\;)a,\theta(X,Y))
		\end{align*}
		
		Then,	$Q(X,a,Y)=0$ if and only if
		\[ \begin{cases}
			\ad^\G_{\theta(X,Y)}=-[F(X),G(Y)]\\
			G(Y)^*\psi(X,Z)-F(X)^*\xi(Y,Z)	+K(Z)^*\theta(X,Y)=0
		\end{cases} \]and we recover the equations 7 and 14 of \eqref{eqmain3}.
		
		\item We have   
		\begin{align*}
			X\circ (a\circ b)&= X\circ[a,b]_\G  
			=F(X)([a,b]_\G)+\om_\G(\psi(X,\;\centerdot\;),[a,b]_\G)\\
			(X\circ a)\circ b&= F(X)a \circ b
			=[F(X)a,b]_\G +\om_\G(K(\;\centerdot\;)F(X)a,b)\\
			a\circ (X\circ b)&= a\circ  F(X)b
			= [a, F(X)b]_\G+ \om_\G(K(\;\centerdot\;)a,F(X)b).
		\end{align*}
		
		Then,	$Q(X,a,b)=0$ if and only if 
		\[ \begin{cases}
			F(X)([a,b]_\G) =[F(X)a,b]_\G +[a, F(X)b]_\G,\\
			\Ri_{\psi(X,Y)}^{\star_\G}+K(Y)F(X)+F(X)^*K(Y)=0
		\end{cases} \]and we recover the fact that $F(X)$ is a derivation and the equation 12 of \eqref{eqmain3}.
		
		\item  We have 
		\begin{align*}
			a\circ (b\circ X)&=a\circ  G(X)b
			=[a, G(X)b]_\G + \om_\G(K(\;\centerdot\;)a,G(X)b)\\
			(a\circ b)\circ X&=[a,b]_\G  \circ X\\
			&= G(X)([a,b]_\G)+\om_\G(\xi(X,\;\centerdot\;),[a,b]_\G)\\
			b\circ (a\circ X)&=[b, G(X)a]_\G + \om_\G(K(\;\centerdot\;)b,G(X)a).
		\end{align*}
		Thus, $Q(a,b,X)=0$ if and only if
		\[ \begin{cases}
			G(X)([a,b]_\G)=[a, G(X)b]_\G+[ G(X)a,b]_\G,\\
			\Ri_{\xi(X,Y)}^{\star_\G}+G(X)^*K(Y)+K(Y)^*G(X)=0
		\end{cases}
		\]
		and we recover the fact that $G(X)$ is a derivation and the equation 13 of \eqref{eqmain3}.
		\item We have    
		\begin{align*}
			a \circ(b\circ c)&=a\circ [b,c]_\G \\
			&=[a,[b,c]_\G]_\G+\om_\G(K(\;\centerdot\;)a,[b,c]_\G)..
		\end{align*}
		Thus, $Q(a,b,c)=0$ is equivalent to  $(\G,\br_\G)$ is a Lie algebra and
		\[ \om_\G(K(X)a,[b,c]_\G)=
		\om_\G(K(X)[a,b]_\G,c)+\om_\G(K(X)b,[a,c]_\G).\]
		This is equivalent to
		\[ K(X)([a,b]_\G)=a\star_\G K(X)b-b\star_\G K(X)a \]and we find the last equation in \eqref{eqmain3}.
		\item We have
		\begin{align*}
			a\circ (X\circ Y)&=a\circ  \theta(X,Y)
			=[a,\theta(X,Y)]+\om_\G(K(\;\centerdot\;)a,\theta(X,Y))\\
			(a\circ X)\circ Y	&= G(X)a\circ Y=G(Y)G(X)a+\om_\G(\xi(Y,\;\centerdot\;),G(X)a),\\
			X\circ (a\circ Y)&=X\circ G(Y)a
			= F(X)G(Y)a+\om_\G(\psi(X,\;\centerdot\;);G(Y)a).
		\end{align*}
		Then, $Q(a,X,Y)=0$ if and only if
		\[ \begin{cases}
			\ad^\G_{\theta(X,Y)}=-F(X)G(Y)-G(Y)G(X),\\
			G(X)^*\xi(Y,Z)+G(Y)^*\psi(X,Z)+K(Z)^*\theta(X,Y)=0.
		\end{cases} \]
		The second equation combined to the equation 7 and the equation 9 in \eqref{eqmain3} gives the equation 8. The equation first equation combined to the equation 14 in \eqref{eqmain3} gives the equation 11 in \eqref{eqmain3}.
		\item We have 
		\begin{align*}
			a\circ (X\circ b)&=a\circ F(X)b
			=[a, F(X)b]_\G+\om_\G(K(\;\centerdot\;)a,F(X)b)\\
			(a\circ X)\circ b&=(G(X)a)\circ b
			= [G(X)a, b]_\G+\om_\G(K(\;\centerdot\;)G(X)a,b)\\
			X\circ (a\circ b)&=X\circ [a,b]_\G 
			=F(X)([a,b]_\G)+\om_\G(\psi(X,\;\centerdot\;),[a,b]_\G).
		\end{align*}
		Thus, $Q(a,X,b)=0$ if and only if 
		\[\begin{cases}
			F(X)([a,b]_\G)=[a, F(X)b]_\G-[G(X)a, b]_\G,\\
			\Ri_{\psi(X,Y)}^{\star_\G}=K(Y)G(X)-F(X)^*K(Y)
		\end{cases}  \]
		Since $F(X)$ is a derivation, the equation $F(X)([a,b]_\G)=[a, F(X)b]_\G-[G(X)a, b]_\G$ means that $\ad^\G_{(F(X)+G(X))a}=0$ and gives the equation 15 in \eqref{eqmain3}. The second equation combined to the equation 12 in \eqref{eqmain3} gives the equation 10. 
		
	\end{enumerate}

\end{document}